\documentclass[10pt,reqno]{amsart}

\title[Global well-posedness for the \MG equation]{Global well-posedness for an advection-diffusion equation arising in magneto-geostrophic dynamics}
\author{Susan Friedlander}
\address{Department of Mathematics,
University of Southern California, 3620 S.~Vermont Ave.,
Los Angeles, CA 90089} \email{\tt susanfri@usc.edu}

\author{Vlad Vicol}
\address{Department of Mathematics, University of Chicago, 5734 University Ave., Chicago, IL 60637}
\email{\tt vicol@math.uchicago.edu}
\usepackage{amsmath, amsthm, amssymb}
\usepackage{color}
\theoremstyle{plain}
\newtheorem{theorem}{Theorem}[section]

\newtheorem{lemma}[theorem]{Lemma}

\newtheorem{corollary}[theorem]{Corollary}

\theoremstyle{definition}
\newtheorem{remark}[theorem]{Remark}

\def\tilde{\widetilde}
\numberwithin{equation}{section}

\newcommand\LL[2]{L_{t}^{#1} L_{x}^{#2}}
\newcommand\LH[1]{L_{t}^{#1} \dot{H}_{x}^{1}}
\newcommand\LP[1]{L_{t,x}^{#1}}
\newcommand\LBMO[1]{L_{t}^{#1} BMO_{x}}
\newcommand\LDBMO[1]{L_{t}^{#1} BMO_{x}^{-1}}
\newcommand\intinttt[3]{\int_{#1}^{#2}\!\!\!\!\int_{#3}}
\newcommand\intint[1]{\int\!\!\!\int_{#1}}

\newcommand\BB[2]{\dot{B}_{#2,\infty}^{#1}}

\renewcommand\hat{\widehat}
\def\ZZ3{{\mathbb Z}^3}
\def\RR1{{\mathbb R}}
\def\RR2{{\mathbb R}^2}
\def\RR3{{\mathbb R}^3}
\def\RRd{{\mathbb R}^d}
\def\TT2{{\mathbb T}^2}
\def\osc{{\rm osc}}

\def\teps{\theta^\epsilon}
\def\ueps{u^\epsilon}

\def\xdomain{ {\mathbb R}^2 \times {\mathbb T}}
\def\Fdomain{ {\mathbb R}^2 \times {\mathbb Z}}

\def\MG{MG\ }

\def\ddiv{\mathop{\rm div} \nolimits}

\def\kkx{k_{1}}
\def\kky{k_{2}}
\def\kkz{k_{3}}
\def\KK{k}

\newcommand\blu[1]{#1}
\newcommand\red[1]{#1}
\newcommand\chd[1]{{#1}}

\begin{document}

%%%%%%%%%%%%%%%%%%%%%%%%% THE ABSTRACT %%%%%%%%%%%%%%%%%%%%%%%%%%%%%%%%%%%

\begin{abstract}
We use De Giorgi techniques to prove H\"older continuity of weak solutions to a class of drift-diffusion equations, with $L^2$ initial data and divergence free drift velocity that lies in $L_{t}^{\infty}BMO_{x}^{-1}$. We apply this result to prove global regularity for a family of active scalar equations which includes the advection-diffusion equation that has been proposed by Moffatt in the context of magnetostrophic turbulence in the Earth's fluid core.
%
%\bigskip
%\noindent {\sc R\'esum\'e.\;} Nous utilisons des techniques de De Giorgi pour d\'emontrer la continuit\'e H\"older de solutions faibles pour une classe d'\'equations de d\'erive-diffusion, avec donn\'ees initiales $L^2 $ et champ de vitesse incompressible appartenant \`a $ L_{t}^{\infty }BMO_{x}^{-1}$. Nous appliquons ce r\'esultat pour d\'emontrer la r\'egularit\'e globale pour une famille d'\'equations du scalaire actif qui comprend l'\'equation d'advection-diffusion qui a \'et\'e propos\'ee par Moffatt dans le contexte de la turbulence magn\'etostrophique dans le noyau fluide de la Terre.
\end{abstract}

%%%%%%%%%%%%%%%%%%%%%%%% Classification and Keywords %%%%%%%%%%%%%%%%%%%%



\maketitle
%%%%%%%%%%%%%%%%%%%%%%%%%% The Main Part %%%%%%%%%%%%%%%%%%%%%%%%%%%%%%%%%%

\section{Introduction}\label{sec:intro}
Active scalar evolution equations have been a topic of considerable study in recent years, in part because they arise in many physical models. In particular, \blu{such} equations are prevalent in fluid dynamics. In this paper we first examine a class of drift-diffusion equations for \blu{an unknown} scalar field $\theta(t,x)$, of the form
\begin{align}\label{eq:intro:1.1}
\partial_t \theta + (v \cdot \nabla)\theta = \Delta \theta,
\end{align}where $v(t,x)$ is a given divergence free vector field that lies in the function space $\LL{2}{2}\cap L_{t}^{\infty}BMO_{x}^{-1}$, \blu{$t>0$, and $x\in\RRd$}. \blu{In~Theorem~\ref{thm:Holder:abstract} we} prove that weak solutions to \eqref{eq:intro:1.1} are H\"older continuous. \blu{Note} that this result is new for such linear parabolic equations with very singular coefficients~\cite{AronsonSerrin,LadySolonnUralceva,Lieberman,Osada,Semenov,Zhang04,Zhang06}. We then use this result to prove in \red{Theorem~\ref{thm:nonlinear}} that Leray-Hopf weak solutions of the active scalar equation
\begin{align}
  &\partial_t \theta + (u\cdot \nabla) \theta = \Delta \theta \label{eq:intro:1.2}\\
  &\ddiv u = 0\\
  &u_j = \partial_i T_{ij} \theta \label{eq:intro:1.3}
\end{align} are classical solutions. \blu{In \eqref{eq:intro:1.3}, the} velocity vector $u$ is obtained from $\theta$ via $\{T_{ij}\}$, a $d\times d$ matrix of Calder\'on-Zygmund singular integral operators (that is, they are bounded $L^2 \mapsto L^2$ and $L^\infty\mapsto BMO$) such that $\partial_i \partial_j T_{ij} \equiv 0$. Note that in \eqref{eq:intro:1.3} we have used the summation convention on repeated indices, and $i,j\in\{1,\ldots,d\}$.

Our motivation for addressing the system \eqref{eq:intro:1.2}-\eqref{eq:intro:1.3} comes from a model proposed by Moffatt~\cite{Moffatt} for magnetostrophic turbulence in the Earth's fluid core. This model is derived from the full magnetohydrodynamic equations (MHD) in the context of a rapidly rotating, density stratified, electrically conducting fluid. After a series of approximations relevant to the geodynamo model, a linear relationship is established between the velocity and magnetic vector fields, and the scalar ``buoyancy'' $\theta$. The sole remaining nonlinearity in the system occurs in the evolution equation for $\theta$, which has the form
\begin{align}
  & \partial_t \theta + (u\cdot \nabla)\theta = S +\kappa \Delta \theta,\label{eq:intro:1.4}
\end{align}
where $S$ is a source term, and $\kappa$ is the coefficient of thermal diffusivity. Here the \red{three dimensional velocity $u$} is such that $\ddiv u =0$, and it is obtained from the buoyancy via
\begin{align}
   u = M[\theta], \label{eq:intro:1.4'}
\end{align}\blu{where $M$ is a nonlocal differential operator of order $1$}. We describe the precise form of the operator $M$ in Section~\ref{sec:MG}. An important feature of this operator is the spatial inhomogeneity that occurs due to the underlying mean magnetic field. We call \eqref{eq:intro:1.4}-\eqref{eq:intro:1.4'} the {\it magnetogeostrophic equation} (MG). We show that the MG system satisfies the conditions under which we prove Theorem~\ref{thm:nonlinear}, and hence obtain (\red{cf.~Theorem~\ref{thm:nonlinearMG}}) global well-posedness for \eqref{eq:intro:1.4}-\eqref{eq:intro:1.4'}.

An active scalar equation that has received much attention in the mathematical literature following its presentation by Constantin, Majda, and Tabak~\cite{ConstMajTab}, as a two-dimensional toy model for the three-dimensional fluid equations, is the so called surface quasi-geostrophic equation (SQG) (see, for example,~\cite{CaffVass,CorCor04,CorFeff,ConstWu08,ConstWu09,KisNazVolb,Silvestre10a,Wu04} and references therein). The dissipative form of this equation for which there is a physical derivation is
\begin{align}
  \partial_t \theta + (u\cdot\nabla)\theta = - (-\Delta)^{1/2} \theta, \label{eq:intro:1.5}
\end{align}where
\begin{align}
  u = \nabla^\perp \chd{(-\Delta)^{-1/2}} \theta \equiv(R_2\theta,-R_1 \theta)\label{eq:intro:1.6}
\end{align}and $R_i$ represents the $i^{th}$ Riesz transform. It was recently proved by Caffarelli and Vasseur~\cite{CaffVass} that solutions of \eqref{eq:intro:1.5}-\eqref{eq:intro:1.6} with $L^2$ initial data are smooth \chd{(see also the review article~\cite{CaffVass2})}. Well-posedness for \eqref{eq:intro:1.5}-\eqref{eq:intro:1.6} in the case of smooth periodic initial data was also obtained by Kiselev, Nazarov, and Volberg~\cite{KisNazVolb}. See also Constantin and Wu~\cite{ConstWu08,ConstWu09} for the super-critically dissipative SQG.

We note that the magnetogeostrophic equation MG and the critically dissipative SQG equation \eqref{eq:intro:1.5}-\eqref{eq:intro:1.6} are both derived from the Navier-Stokes equations in the context of a rapidly rotating fluid in a thin shell. For both systems the Coriolis force is dominant in the momentum equation. In the case of the SQG equation the relation \eqref{eq:intro:1.6} is derived via a projection of the three-dimensional problem onto the two-dimensional \red{horizontal bounding surface}. In the case of the MG equation the coupling with the magnetic induction equation closes the three-dimensional linear system that produces the operators $\{T_{ij}\}$, with $u_j = \partial_i T_{ij} \theta$.

Systems \eqref{eq:intro:1.2}-\eqref{eq:intro:1.3} and \eqref{eq:intro:1.5}-\eqref{eq:intro:1.6} have strong similarities. In particular, they have the same relative order of the spatial derivatives between the advection term and the diffusive term. Moreover, if $\theta(t,x)$ is a solution of \eqref{eq:intro:1.2}-\eqref{eq:intro:1.3}, then $\theta_\lambda(t,x) = \theta(\lambda^2 t, \lambda x)$ is \red{also} a solution, and hence $L^\infty(\RRd)$ is the critical Lebesgue space with respect to the natural scaling of the equation. We note that $L^\infty$ is also the critical Lebesgue space for the critically dissipative surface quasi-geostrophic equation \eqref{eq:intro:1.5}-\eqref{eq:intro:1.6}, and for the modified surface quasi-geostrophic equation (cf.~Constantin, Iyer, and Wu~\cite{ConstIyerWu}). The advantage of system \eqref{eq:intro:1.2}-\eqref{eq:intro:1.3} over the critical SQG equation is that the diffusive term is given via a local operator. The tradeoff is that the drift velocity in \eqref{eq:intro:1.2}-\eqref{eq:intro:1.3} is \blu{more singular, i.e.,} the derivative of a $BMO$ function \red{(see Koch and Tataru~\cite{KochTat} for the Navier-Stokes equations in $BMO^{-1}$).}

Our proof of Theorem~\ref{thm:Holder:abstract} and Theorem~\ref{thm:nonlinear} is along the lines of the proof of \red{Caffarelli and Vasseur~\cite[Theorem 3]{CaffVass}} for the critical SQG equation. The primary technique employed in \cite{CaffVass,ConstWu09,Vass07}, and in the present paper, is the De Giorgi iteration~\cite{DeGiorgi}. This consists of first showing that a weak solution is bounded by proving that the function $\max\{\theta-h,0\}$ has zero energy if $h$ is chosen large enough. Then a diminishing oscillation result implies smoothness of the solution in a subcritical space, namely $C^\alpha$, for some $\alpha \in (0,1)$. \red{The} proof of H\"older continuity for solutions of \eqref{eq:intro:1.1} with $v\in \LDBMO{\infty}$ does not follow directly \red{either} from \cite{CaffVass}, where $v\in\LBMO{\infty}$, \red{or} from \cite{ConstIyerWu}, where $v\in L_{t}^{\infty} C_{x}^{1-\alpha}$ and \chd{$\alpha\in(0,1)$}. The crucial step in the proof of Theorem~\ref{thm:Holder:abstract} is the local energy and uniform estimates. The main obstruction to applying the classical parabolic De Giorgi estimates via an $L^p$-based Caccioppoli inequality ($1<p<\infty$), is that $v(t,\cdot) \in BMO^{-1}$. In Section~\ref{sec:abstract} we give details as to how we overcome this \red{difficulty.}

Equation \eqref{eq:intro:1.1} is in the class of parabolic equations in divergence form that have been studied extensively, including in the classical papers of Nash~\cite{Nash}, Moser~\cite{Moser}, Aronson and Serrin~\cite{AronsonSerrin}. Osada~\cite{Osada} allowed for singular coefficients and proved H\"older continuity of solutions to \eqref{eq:intro:1.1} when $v\in L_{t}^{\infty} W_{x}^{-1,\infty}$ is divergence free. Hence Theorem~\ref{thm:Holder:abstract} may be also viewed as an improvement of the results of Osada, since if $f \in BMO(\RRd) \cap L^2(\RRd)$, then it does not follow that $f\in L^\infty(\RRd)$ (cf.~\cite{Stein93}). In the same spirit, Zhang~\cite{Zhang04,Zhang06} and Semenov~\cite{Semenov} give strong regularity results for parabolic equations of the type \eqref{eq:intro:1.1}, where the singular divergence free velocity satisfies a certain form boundedness condition. We note that this form boundedness condition does not cover the case $v\in\LDBMO{\infty}$, and hence Theorem~\ref{thm:Holder:abstract} does not follow from the results in \cite{Semenov,Zhang04,Zhang06}, and vice-versa. The overall conclusion of the body of work on parabolic equations with a singular drift velocity is that the divergence free structure of $v$ produces a dramatic gain in regularity of the solution, compared to the classical theory (cf.~\cite{LadySolonnUralceva}).

{\bf Organization of the paper.} In Section~\ref{sec:abstract} we prove H\"older regularity for the linear drift-diffusion equation \eqref{eq:intro:1.1}, with $v$ being a given divergence free vector field in the function space $\LP{2} \cap L_{t}^{\infty}BMO_{x}^{-1}$. In Section~\ref{sec:nonlinear} we apply this result to prove that a Leray-Hopf weak solution $\theta$ of the nonlinear active scalar system \eqref{eq:intro:1.2}-\eqref{eq:intro:1.3} is H\"older smooth for positive time. Since H\"older regularity is subcritical for the natural scaling of \eqref{eq:intro:1.2}-\eqref{eq:intro:1.3} we can bootstrap to prove higher regularity and hence conclude that the solution is a classical solution. In Section~\ref{sec:MG} we describe an active scalar equation that arises as a model for magneto-geostrophic dynamics in the Earth's fluid core. We show that this three dimensional {\it MG equation} is an example of the general system \eqref{eq:intro:1.2}-\eqref{eq:intro:1.3}. In the Appendix we prove the existence of weak solutions to \eqref{eq:intro:1.2}-\eqref{eq:intro:1.3} evolving from $L^2({\mathbb R}^d)$ initial data.

\section{Regularity for a parabolic equation with singular drift} \label{sec:abstract}

Consider the evolution of an unknown scalar $\theta(t,x)$ given by
\begin{align}
\partial_t \theta+ (v \cdot \nabla) \theta =  \Delta \theta  \label{eq:theta:1}
\end{align}where the velocity vector $v(t,x) = (v_1(t,x),\ldots,v_d(t,x)) \in \chd{L^2((0,\infty)\times {\mathbb R}^d)}$ is given, and $(t,x) \in  [0,\infty)\times \RRd $. Additionally let $v$ satisfy
\begin{align}
&\partial_j v_j (t,x)= 0 \label{eq:u:divfree}
\end{align}in the sense of distributions. We express $v_j$ as
\begin{align}
&v_j(t,x) = \partial_i V_{ij} (t,x) \label{eq:u:Vdef}
\end{align}in $ [0,\infty) \times\RRd $, where we have used the summation convention on repeated indices, and we denoted $V_{ij} = - (-\Delta)^{-1} \partial_i v_j$. The matrix $\{V_{ij}\}_{i,j=1}^{d}$ is given, and satisfies
\begin{align}
  V_{ij} \in L^\infty((0,\infty);L^2(\RRd)) \cap L^2((0,\infty); \dot{H}^1(\RRd)) \label{eq:V:integrability}
\end{align}for all $i,j\in\{1,\ldots,d\}$.

\begin{theorem}[\bf The linear problem]\label{thm:Holder:abstract}
Given $\theta_0 \in L^2(\RRd)$ and $\{V_{ij}\}$ satisfying \eqref{eq:V:integrability}, let $\theta \in L^\infty([0,\infty);L^2(\RRd)) \cap L^2( (0,\infty);\dot{H}^1(\RRd))$ be a global weak solution of the initial value problem associated to \eqref{eq:theta:1}--\eqref{eq:u:Vdef}. If additionally we have $V_{ij} \in L^\infty([t_0,\infty);BMO(\RRd))$ for all  $i,j\in\{1,\ldots,d\}$ and some $t_0 > 0$, then there exists $\alpha>0$ such that $\theta \in C^\alpha([t_0,\infty)\times \RRd)$.
\end{theorem}

In analogy with the constructions in \cite{CaffVass,ConstWu09}, the proof of Theorem~\ref{thm:Holder:abstract} consists of two steps. For $t_0 >0$ fixed, we first prove that $\theta \in L^{\infty}([t_0,\infty);L^{\infty}(\RRd))$. The main challenge is to prove the H\"older regularity of the solution, which is achieved by using the method of De Giorgi iteration (cf.~\cite{DeGiorgi,Giaq,Lieberman}). Note that for divergence-free $v\in \LP{2}$, the existence of a weak solution $\theta$ to \eqref{eq:theta:1}--\eqref{eq:u:Vdef}, evolving from $\theta_0 \in L^2$, is known (for instance, see~\cite{Semenov} where the more general $v\in L_{loc}^{1}$ is treated, also~\cite{CaffVass}, and references therein). Moreover, this weak solution satisfies the classical energy inequality and the level set energy inequalities~\eqref{eq:L^2energy} below.

\begin{remark}
  The conclusion of Theorem~\ref{thm:Holder:abstract} holds if the Laplacian on the right side of \eqref{eq:theta:1} is replaced by a generic second-order strongly elliptic operator \blu{$\partial_i( a_{ij} \partial_{j})$}, with bounded measurable coefficients $\{a_{ij}\}$.
  %See also~\cite{DongKim} and references therein for the case when $a_{ij} \in VMO(\RRd)$.
\end{remark}

\begin{remark}
\chd{We note that the De Giorgi techniques used here
to prove H\"older regularity for solutions to \eqref{eq:theta:1}--\eqref{eq:u:Vdef} can also
be used to prove H\"older regularity for the problem with a
forcing term $S$ on the right side of \eqref{eq:theta:1}. In this case we consider
$S \in L_{t,x}^r$ to be an externally given force, with $r>1+d/2$ (cf.~\cite{Lieberman}).}
\end{remark}

\begin{remark}
  \chd{In a very recent preprint, Seregin, Silvestre, \v{S}ver\'ak, and Zlato\v{s}~\cite{SSSZ} also use De Giorgi techniques to prove H\"older regularity of solutions to a parabolic equation with drift velocities in $L_t^\infty BMO_x^{-1}$.}
\end{remark}

{\bf Notation.} In the following we shall use the classical function spaces: $L^p$ - Lebesgue spaces, $BMO$ - functions with bounded mean oscillation, $BMO^{-1}$ - derivatives of $BMO$ functions, $\dot{H}^s$ - homogeneous Sobolev spaces, and $C^\alpha$ - H\"older spaces. To emphasize the different integrability in space and time we shall denote $L^p([0,\infty);L^q(\RRd))$ by $\LL{p}{q}$ for $1\leq p,q\leq \infty$, and similarly for $\LH{p}$ and $\LBMO{p}$. Also $\LP{p}(I \times B) = L^p(I;L^p(B))$ for any $I\subset {\mathbb R}$ and $B \subset {\mathbb R}^d$. The ball in ${\mathbb R}^d$ and the parabolic cylinder in ${\mathbb R}^{d+2}$ are classically denoted by $B_\rho(x_0) = \{ x\in \RRd \colon |x-x_0|<\rho\}$ and $Q_\rho (t_0,x_0) = [t_0 - \rho^2,t_0] \times B_\rho(x_0)$ for $\rho>0$. Lastly, we shall write $(f - k)_+ = \max\{f-k,0\}$.
\subsection{Boundedness of the solution}
The first step is to show that a weak solution is bounded for \blu{positive time}.
\begin{lemma}[\bf From $L^2$ to $L^\infty$]\label{lemma:L^infty}
Let $\theta \in L^\infty([0,\infty);L^2(\RRd)) \cap L^2( (0,\infty);\dot{H}^1(\RRd))$ be a global weak solution of \eqref{eq:theta:1}-\eqref{eq:u:Vdef} evolving from $\theta_0\in L^{2}(\RRd)$, where $v \in L^2((0,\infty);L^2(\RRd))$. Then for all $t>0$ we have
\begin{align}
  \Vert \theta(t,\cdot) \Vert_{L^\infty(\RRd)} \leq \frac{C \Vert \theta_0 \Vert_{L^2(\RRd)}}{t^{d/4}},
\end{align}for some sufficiently large positive dimensional constant $C$.
\end{lemma}
\begin{proof}
  The proof of this lemma is mutatis-mutandis as in~\cite{CaffVass,ConstWu09}, and requires only the fact that $v$ is divergence free. The main idea is that since $y\mapsto(y-h)_+$ is convex, for all $h>0$ we have
\begin{align*}
  \partial_t (\theta - h)_+ - \Delta(\theta-h)_+ +(v\cdot\nabla)(\theta-h)_+ \leq 0,
\end{align*}and hence, multiplying by $(\theta-h)_+$ integrating by parts, and using that $\ddiv v=0$, we obtain the energy inequality
\begin{align}\label{eq:L^2energy}
  \int_{\RRd} | (\theta(t_2,\cdot) - h)_+|^2 dx + 2 \intinttt{t_1}{t_2}{\RRd} |\nabla (\theta - h)_+|^2 dx dt \leq \int_{\RRd} | (\theta(t_1,\cdot) - h)_+|^2 dx,
\end{align}for all $h>0$ and $0<t_1<t_2<\infty$. For $t_0 >0$, and $H>0$ to be chosen sufficiently large, we define $t_n = t_0 - t_0/2^n$, $h_n = H - H/2^n$, and
\begin{align*}
  c_n = \sup_{t\geq t_n} \int_{\RRd} |(\theta(t,\cdot) - h_n)_+|^2 dx + 2 \intinttt{t_n}{\infty}{\RRd} |\nabla (\theta-h_n)|^2 dx dt,
\end{align*}where $n\geq 0$.
The inequality \eqref{eq:L^2energy}, the Gagliardo-Nirenberg-Sobolev inequality, and Riesz interpolation then imply that
\begin{align*}
  c_{n+1} \leq \chd{\frac{C}{t_0 H^{4/d}}} 2^{n(1+4/d)} c_n^{1+2/d}.
\end{align*}Letting $H = C c_0^{1/2}/t_0^{d/4} \leq C \Vert \theta_0 \Vert_{L^2(\RRd)}/t_0^{d/4}$, for some sufficiently large dimensional constant $C$, implies that $c_n \rightarrow 0$ exponentially as $n\rightarrow \infty$, and therefore \blu{$ \theta(t_0,\cdot) \leq H$. Applying the same procedure to $-\theta$} concludes the proof of the lemma. We refer the reader to~\cite{CaffVass,ConstWu09} for further details.
\end{proof}

\subsection{Local energy and uniform inequalities}
In proving the boundedness of the solution we only required that $v \in \LP{2}$, and $\ddiv v=0$. For the rest of the section we use the additional assumption $v\in \LDBMO{\infty}$.

\begin{lemma}[\bf First energy inequality]\label{lemma:1st:energyineq}
Let $\theta\in\LL{\infty}{2}\cap\LH{2}$ be a global weak solution of the initial value problem associated to \eqref{eq:theta:1}--\eqref{eq:u:Vdef}. Furthermore, assume that $V_{ij} \in L^{\infty}((0,\infty);BMO(\RRd))$ for all $i,j\in\{1,\ldots,d\}$, and \eqref{eq:V:integrability} holds. Then for any $0<r<R$ and \blu{$h\in {\mathbb R}$}, we have
\begin{align}\label{eq:1st:energyineq}
&\Vert (\theta-h)_+ \Vert_{\LL{\infty}{2}(Q_r)}^2 + \Vert \nabla (\theta-h)_+ \Vert_{\LP{2}(Q_r)}^2\notag\\
& \qquad \qquad \leq \frac{C\, R}{(R-r)^2} \Vert (\theta-h)_+ \Vert_{\LP{2}(Q_R)}^{2-\frac{2}{d+2}} \Vert (\theta -h)_+ \Vert_{\LP{\infty}(Q_R)}^{\frac{2}{d+2}},
\end{align}where $C = C(d,\Vert V_{ij} \Vert_{\LBMO{\infty}})$ is a fixed positive constant, and we have denoted $Q_\rho = [t_0 - \rho^2,t_0]\times B_\rho(x_0)$ for $\rho>0$ and an arbitrary $(t_0,x_0) \in (0,\infty)\times \RRd$. \blu{Moreover, estimate \eqref{eq:1st:energyineq} also holds with $\theta$ replaced by $-\theta$.}
\end{lemma}
\begin{remark}
  Note that from Lemma~\ref{lemma:L^infty} we have that $\theta \in \LP{\infty}$, and hence the right side of \eqref{eq:1st:energyineq} is finite.
\end{remark}
\begin{remark}
The classical local energy inequality (cf.~\cite{Giaq,Lieberman,Osada}, see also~\cite{CaffVass,ConstWu09}) does not contain the term $\Vert (\theta -h)_+ \Vert_{\LP{\infty}(Q_R)}$ on the right, since the velocity field $v$ is not as singular as in our case. In this section we prove that since in \eqref{eq:1st:energyineq} the exponent $2/(d+2)$ of $\Vert (\theta -h)_+ \Vert_{\LP{\infty}(Q_R)}$ is ``small enough'', the De Giorgi program may still be carried out to obtain the H\"older regularity of weak solutions.
\end{remark}
\begin{proof}[Proof of Lemma~\ref{lemma:1st:energyineq}]
Fix $h\in {\mathbb R}$ and let $0<r<R$ be such that $t_0/2 - R^2 > 0$. Let $\eta(t,x)\in C_{0}^{\infty}((0,\infty)\times \RRd)$ be a smooth cutoff function such that
\begin{align*}
  & 0 \leq \eta \leq 1\ \mbox{in}\ (0,\infty)\times \RRd,\\
  & \eta \equiv 1\ \mbox{in}\ Q_r(x_0,t_0), \ \mbox{and}\ \eta \equiv 0\ \mbox{in}\ {\mathop cl} \{Q_R^c(x_0,t_0) \cap \{ (t,x): t\leq t_0\}\},\\
  & |\nabla \eta| \leq \frac{C}{R-r}, |\nabla \nabla \eta| \leq \frac{C}{(R-r)^2}, |\partial_t \eta| \leq \frac{C}{(R-r)^2}\ \mbox{in}\ Q_R(x_0,t_0) \setminus Q_r(x_0,t_0),
\end{align*}for some positive dimensional constant $C$. Define $t_1 = t_0 - R^2 > 0$ and let $t_2 \in [t_0-r^2,t_0]$ be arbitrary. Multiply \eqref{eq:theta:1} by $(\theta-h)_+ \eta^2$ and then integrate on $[t_1,t_2]\times \RRd$ to obtain
\begin{align}\label{eq:local:1}
  \intinttt{t_1}{t_2}{\RRd} \partial_t \left( (\theta - h)_+^2\right) \eta^2\; dx dt &-2 \intinttt{t_1}{t_2}{\RRd} \partial_{jj} (\theta-h)_+ (\theta-h)_+ \eta^2\; dx dt\notag\\
   &\qquad + \intinttt{t_1}{t_2}{\RRd} \partial_i V_{ij}\; \partial_j \left( (\theta-h)_+^2\right) \eta^2 \; dx dt = 0.
\end{align}The main obstruction to applying classical the de Giorgi estimates (via the $L^p$-based Caccioppoli inequality, cf.~\cite{Giaq,Lieberman}) is that $\partial_i V_{ij}\in L_t^\infty BMO_x^{-1}$, as opposed to the case $L_t^\infty W_x^{-1,\infty}$ considered by Osada~\cite{Osada} (see also~\cite{Semenov}). We overcome this difficulty by subtracting from $V_{ij}(t,\cdot)$ its spatial mean over $\{t\}\times B_R$, namely $\overline{V}_{ij,B_R}(t)$ (this does not introduce any lower order terms because $\partial_{x_i} \overline{V}_{ij,B_R}(t)=0$), and by appealing to the John-Nirenberg inequality. More precisely, we define
\begin{align}\label{eq:V:mean:def}
\tilde{V}_{ij,R}(t,x) = V_{ij}(t,x) - \overline{V}_{ij,B_R}(t) = V_{ij}(t,x) - \frac{1}{|B_R|} \int_{B_R} V_{ij}(t,y)\; dy,
\end{align}and note that $\partial_i V_{ij} = \partial_i \tilde{V}_{ij,R}$. Therefore, the third term on the left of \eqref{eq:local:1} may be replaced by
\begin{align*}
  \intinttt{t_1}{t_2}{\RRd} \partial_i \tilde{V}_{ij,R}\; \partial_j \left( (\theta-h)_+^2\right) \eta^2 \; dx dt.
\end{align*}We integrate by parts in $t$ the first term on the left of \eqref{eq:local:1}, and use $\eta(t_1,\cdot)\equiv 0$. The second term we integrate twice by parts in $x_j$, and the third term on the left of \eqref{eq:local:1} we integrate by parts first in $x_j$ (and use $\partial_j (\partial_i \tilde{V}_{ij,R}) = \partial_i (\partial_j V_{ij}) = \partial_j v_j = 0$) and then integrate by parts in $x_i$, to obtain
\begin{align}\label{eq:local:2}
& \frac{1}{2}\int_{\RRd} (\theta(t_2,\cdot)-h)_+^2 \eta(t_2,\cdot)^2 \; dx  +  \intinttt{t_1}{t_2}{\RRd} | \nabla (\theta-h)_+|^2 \eta^2 \; dx dt\notag\\
& \qquad \qquad =  \intinttt{t_1}{t_2}{\RRd} (\theta-h)_+^2 \eta \partial_t \eta\; dx dt + \intinttt{t_1}{t_2}{\RRd} (\theta- h)_+^2 \partial_j (\eta \partial_j \eta)\; dx dt\notag\\
& \qquad \qquad \qquad - \intinttt{t_1}{t_2}{\RRd} \tilde{V}_{ij,R} (\theta-h)_+^2 \partial_i (\eta \partial_j \eta) \; dx dt\notag\\
&\qquad \qquad \qquad - 2 \intinttt{t_1}{t_2}{\RRd} \tilde{V}_{ij,R} \partial_i (\theta-h)_+ (\theta-h)_+ \eta \partial_j \eta \; dx dt .
\end{align}
Using the bounds on the time and space derivatives of $\eta$, the fact that $\eta \equiv 1$ on $Q_r$, $t_2\leq t_0$, the H\"older and $\varepsilon$-Young inequalities, \red{we obtain from \eqref{eq:local:2}}
\begin{align}\label{eq:local:3}
  & \int_{B_r} (\theta(t_2,\cdot)-h)_+^2 \; dx + 2 \intinttt{t_1}{t_2}{\RRd} |\nabla (\theta-h)_+|^2 \eta^2 \; dx dt \notag\\
  & \qquad \leq \frac{C}{(R-r)^2} \intint{Q_R} (\theta-h)_+^2 \; dx dt + \frac{C}{(R-r)^2} \intint{Q_R} | \tilde{V}_{ij,R} |\; (\theta-h)_+^2\; dxdt \notag\\
  & \qquad + \intinttt{t_1}{t_2}{\RRd} |\nabla (\theta-h)_+|^2 \eta^2 \; dx dt  + \frac{C}{(R-r)^2} \intint{Q_R} | \tilde{V}_{ij,R} |^2 (\theta-h)_+^2\; dxdt.
\end{align}
After absorbing the third term on the right of \eqref{eq:local:3} into the left side, we take the supremum over $t_2 \in [t_0 - r^2,t_0]$, to obtain
\begin{align}\label{eq:local:4}
& \Vert (\theta-h)_+ \Vert_{\LL{\infty}{2}(Q_r)}^2 + \Vert \nabla (\theta-h)_+ \Vert_{\LP{2}(Q_r)}^2\notag\\
& \qquad \leq  \frac{C}{(R-r)^2} \Vert (\theta-h)_+\Vert_{\LP{2}(Q_R)}^2 + \frac{C}{(R-r)^2} \intint{Q_R} | \tilde{V}_{ij,R} |\; (\theta-h)_+^2\; dxdt\notag\\
& \qquad \qquad + \frac{C}{(R-r)^2} \intint{Q_R} | \tilde{V}_{ij,R} |^2 (\theta-h)_+^2\; dxdt.
\end{align} As a corollary of the celebrated John-Nirenberg inequality (cf.~\cite{Giaq,Stein93}) we have that for any fixed $R>0$, $t\in[t_0 - R^2,t_0]$, and $1< p < \infty$,
\begin{align*}
  \Vert \tilde{V}_{ij,R}(t,\cdot) \Vert_{L^p(B_R)} &= \Vert V_{ij}(t,\cdot) - \overline{V}_{ij,B_R}(t) \Vert_{L^p(B_R)}\notag\\
  & \leq C \Vert V_{ij}(t,\cdot) \Vert_{BMO(\RRd)} |B_R|^{1/p},
\end{align*}where $C = C(d,p)>0$ is a fixed constant (recall that $C(d,p)\rightarrow \infty$ as $p\rightarrow \infty$). The fact that $V_{ij} \in L^\infty([t_0/2,\infty); BMO(\RRd))$ implies that for all $t\in [t_0-R^2,t_0]$ we have
\begin{align}\label{eq:V:BMO:bound}
  \Vert \tilde{V}_{ij,R}(t,\cdot) \Vert_{L^p(B_R)} &\leq C_0 |B_R|^{1/p}
\end{align}for a positive constant \chd{$C_0 = C_0(\Vert V_{ij} \Vert_{L^\infty([t_0/2,\infty);BMO(\RRd))},d,p)$}. We fix $0<\varepsilon<2$ to be chosen later, and using \eqref{eq:V:BMO:bound} and the H\"older inequality we bound
\begin{align*}
  \intint{Q_R} | \tilde{V}_{ij,R} |\; (\theta-h)_+^2\; dxdt &= \int_{t_0-R^2}^{t_0} \left(\int_{B_R}| \tilde{V}_{ij,R}(t,x) |\; (\theta-h)_+^2(t,x)  \; dx \right)dt\notag\\
  &\leq C_0 |B_R|^{\varepsilon/2} \int_{t_0-R^2}^{t_0} \Vert (\theta(t,\cdot)-h)_+ \Vert_{L^{4/(2-\varepsilon)}(B_R)}^2\; dt.
\end{align*}Using the interpolation inequality $\Vert f \Vert_{L^p}\leq C \Vert f \Vert_{L^2}^{2/p} \Vert f \Vert_{L^\infty}^{1-2/p}$, with \chd{$p=4/(2-\varepsilon)$}, we obtain from the above estimate that
\begin{align}
  &\intint{Q_R} | \tilde{V}_{ij,R} |\; (\theta-h)_+^2\; dxdt \notag\\
  & \qquad \leq C_0 |B_R|^{\varepsilon/2}\int_{t_0-R^2}^{t_0} \Vert (\theta(t,\cdot)-h)_+ \Vert_{L^{2}(B_R)}^{2-\varepsilon} \Vert (\theta(t,\cdot)-h)_+ \Vert_{L^{\infty}(B_R)}^\varepsilon\; dt\notag\\
  & \qquad \leq C_0 R^{\varepsilon(\chd{d}+2)/2} \Vert (\theta-h)_+ \Vert_{\LP{2}(Q_R)}^{2-\varepsilon} \Vert (\theta-h)_+ \Vert_{\LP{\infty}(Q_R)}^\varepsilon.\label{eq:1stBMObound}
\end{align}Similarly, from \eqref{eq:V:BMO:bound}, the H\"older inequality and $L^p$ interpolation, we obtain
\begin{align}
  \intint{Q_R} | \tilde{V}_{ij,R} |^2  (\theta-h)_+^2\; dxdt \leq C_0 R^{\varepsilon(\chd{d}+2)/2} \Vert (\theta-h)_+ \Vert_{\LP{2}(Q_R)}^{2-\varepsilon} \Vert (\theta-h)_+ \Vert_{\LP{\infty}(Q_R)}^\varepsilon.\label{eq:2ndBMObound}
\end{align}Combining estimates \eqref{eq:local:4} with \eqref{eq:1stBMObound}, \eqref{eq:2ndBMObound}, and the H\"older inequality, we conclude that
\begin{align}\label{eq:finalBMObound}
& \Vert (\theta-h)_+ \Vert_{\LL{\infty}{2}(Q_r)}^2 + \Vert \nabla (\theta-h)_+ \Vert_{\LP{2}(Q_r)}^2\notag\\
& \qquad \leq \frac{C_0 R^{\varepsilon(\chd{d}+2)/2}}{(R-r)^2} \Vert (\theta-h)_+ \Vert_{\LP{2}(Q_R)}^{2-\varepsilon} \Vert (\theta-h)_+ \Vert_{\LP{\infty}(Q_R)}^\varepsilon.
\end{align}The proof of the lemma is concluded by letting $\varepsilon = 2/(d+2)$ in \eqref{eq:finalBMObound} above.
\end{proof}
By applying the H\"older inequality to the right side of \eqref{eq:1st:energyineq} we then obtain:
\begin{corollary}\label{corr:energy}
  Let $\theta$ be as in Lemma~\ref{lemma:1st:energyineq}. Then we have
  \begin{align}\label{eq:energy:corrolary}
   \Vert (\theta-h)_+ \Vert_{\LL{\infty}{2}(Q_r)}^2 + \Vert \nabla (\theta-h)_+ \Vert_{\LP{2}(Q_r)}^2 \leq \frac{C R^{d+2}}{(R-r)^2} \Vert (\theta-h)_+ \Vert_{\LP{\infty}(Q_R)}^2,
  \end{align}for some positive constant $C = C(d,\Vert V_{ij} \Vert_{\LBMO{\infty}})$.
\end{corollary}We now fix a point $(t_0,x_0)\in (0,\infty)\times\RRd$ and we prove the H\"older continuity of $\theta$ at this point. Throughout the following we denote by $Q_\rho$ the cylinder $Q_\rho(t_0,x_0)$, for any $\rho>0$.

The following lemma gives an estimate on the supremum of \blu{$\theta$} on a half cylinder, in terms of the supremum on the full cylinder. A similar statement may be proven for $-\theta$.
\begin{lemma}\label{lemma:sup:halfball}
  Let $\theta$ be as in Lemma~\ref{lemma:1st:energyineq}. Assume that \blu{$h_0 \leq \sup_{Q_{r_0}}\theta$}, where $r_0 >0$ is arbitrary. We have
  \begin{align}\label{eq:sup:halfball}
  \sup_{Q_{r_0/2}}\blu{\theta} \leq h_0 + C \left( \frac{|\{\theta > \blu{h_0}\}\cap Q_{r_0}|^{1/(d+2)}}{r_0}\right)^{1/2} \left( \sup_{Q_{r_0}} \blu{\theta} - h_0 \right)
  \end{align}for some positive constant $C = C(d,\Vert V_{ij} \Vert_{\LBMO{\infty}})$.
\end{lemma}The above estimate differs from the classical one cf.~\cite[Theorem 6.50]{Lieberman} in that the power of $|\{\theta > \blu{h_0}\}\cap Q_{r_0}|/|Q_{r_0}|$ is $1/(2d+4)$ instead of $1/(d+2)$. However, the key feature of \eqref{eq:sup:halfball} is that the coefficient of \red{($\sup_{Q_{r_0}} \theta - h_0$)} does not scale with $r_0$. It is convenient to introduce the following notation:
\begin{itemize}
\item $A(h,r) = \{ \theta > h\} \cap Q_r$
\item $a(h,r) = |A(h,r)|$
\item $b(h,r) = \Vert (\theta - h)_+ \Vert_{\LP{2}(Q_r)}^2$
\item \blu{$M(r) = \sup_{Q_r} \theta$}
\item \blu{$m(r) = \inf_{Q_r} \theta$}
\item \blu{$\osc(Q) = \sup_{Q} \theta - \inf_Q \theta$}
\end{itemize}
\begin{proof}[Proof of Lemma~\ref{lemma:sup:halfball}]
Let $0<r<R$ and $0<h<H$. We have
\begin{align}\label{eq:9}
b(h,r) = \Vert \theta-h \Vert_{\LP{2}(A(h,r))}^2 \geq \Vert \theta-h \Vert_{\LP{2}(A(H,r))}^2 \geq (H-h)^2 a(H,r).
\end{align}Let $\eta(t,x) \in C_0^\infty({\mathbb R} \times\RRd)$ be a smooth cutoff such that $\eta \equiv 1$ on $Q_r$, $\eta \equiv 0$ on $Q_{(r+R)/2}^c \cap \{t\leq t_0\}$, and $|\nabla \eta| \leq C/(R-r)$ for some universal constant $C>0$. Then, by H\"older's inequality and the choice of $\eta$ we obtain
\begin{align}
  b(h,r) = \Vert (\theta-h)_+ \Vert_{\LP{2}(Q_r)}^2 &\leq a(h,r)^{2/(d+2)} \Vert (\theta-h)_+ \Vert_{\LP{2(d+2)/d}(Q_r)}^2 \notag \\
  &\leq a(h,r)^{2/(d+2)} \Vert \eta (\theta-h)_+ \Vert_{\LP{2(d+2)/d}((-\infty,t_0)\times\RRd)}^2 . \label{eq:B&A}
\end{align}Using the Gagliardo-Nirenberg-Sobolev inequality and Riesz interpolation
\begin{align*}
\Vert f \Vert_{L^{2(d+2)/d}( (-\infty,t_0)\times \RRd)}^2 \leq C \Vert  f \Vert_{\LL{\infty}{2}((-\infty,t_0)\times \RRd)}^2 + C \Vert \chd{\nabla} f \Vert_{\LP{2}((-\infty,t_0)\times \RRd)}^2,
\end{align*}estimate \eqref{eq:B&A} implies that
\begin{align*}
  b(h,r) &\leq C a(h,r)^{2/(d+2)} \Big( \Vert \eta (\theta-h)_+ \Vert_{\LL{\infty}{2}((-\infty,t_0)\times \RRd)}^2\notag\\
  & \qquad \qquad  \qquad \qquad \qquad + \Vert \nabla \big(\eta (\theta-h)_+ \big) \Vert_{\LP{2}((-\infty,t_0)\times \RRd)}^2 \Big)\notag\\
  & \leq C a(h,r)^{2/(d+2)} \Big( \Vert (\theta-h)_+ \Vert_{\LL{\infty}{2}(Q_{(r+R)/2})}^2 + \Vert \nabla (\theta-h)_+ \Vert_{\LP{2}(Q_{(r+R)/2})}^2  \notag\\
  & \qquad \qquad  \qquad \qquad \qquad +  \frac{1}{(R-r)^2} \Vert (\theta-h)_+ \Vert_{\LP{2}(Q_{(r+R)/2})}^2\Big)
\end{align*}for some positive dimensional constant $C$. Using the first energy inequality, i.e., Lemma~\ref{lemma:1st:energyineq}, and the H\"older inequality, we bound the far right side of the above and obtain
\begin{align}
  b(h,r) &\leq C a(h,r)^{2/(d+2)}  \frac{R}{(R-r)^2} \Vert (\theta-h)_+\Vert_{\LP{2}(Q_R)}^{2- 2/(d+2)} \Vert (\theta-h)_+ \Vert_{\LP{\infty}(Q_R)}^{2/(d+2)}\notag\\
  & \leq C a(h,r)^{2/(d+2)}  \frac{R}{(R-r)^2} b(h,R)^{1 - 1/(d+2)} \Vert (\theta-h)_+\Vert_{\LP{\infty}(Q_R)}^{2/(d+2)}\label{eq:10}
\end{align}
for some sufficiently large positive constant $C=C(d,\Vert V_{ij} \Vert_{\LBMO{\infty}})$. By combining estimates \eqref{eq:9} and \eqref{eq:10} we obtain the main consequence of Lemma~\ref{lemma:1st:energyineq}, that is
\begin{align}\label{eq:11}
b(H,r) \leq \frac{C\; R}{(H-h)^{4/(d+2)} (R-r)^2}\; b(h,R)^{1+1/(d+2)} \Vert (\theta-H)_+\Vert_{\LP{\infty}(Q_R)}^{2/(d+2)}.
\end{align}The above estimates give the proof of the lemma as follows. Let \red{$r_n = r_0/2 + r_0/2^{n+1} \searrow r_0/2$}, $h_n = h_\infty - (h_\infty - h_0)/2^n\nearrow h_\infty$, and $b_n = b(h_n,r_{n+1})$, for all $n\geq0$, where $r_0$ and $h_0$ are as in the statement of the lemma, while $h_\infty>0$ is to be chosen later. By letting $H=h_{n+1}, h=h_n, r=r_{n+2}$, and $R=r_{n+1}$ in \eqref{eq:11}, we obtain
\begin{align}\label{eq:13}
  b_{n+1} &\leq \frac{C r_{n+1}}{(h_\infty-h_0)^{4/(d+2)}r_0^2}\; 2^{n(2 + 4/(d+2))} b_n^{1 + 1/(d+2)} \Vert (\theta-h_{n+1})_+ \Vert_{\LP{\infty}(Q_{r_{n+1}})}^{2/(d+2)}\notag\\
  & \leq \frac{C (M(r_0) - h_0)^{2/(d+2)}}{(h_\infty-h_0)^{4/(d+2)}r_0}\; 2^{n(2 + 4/(d+2))} b_n^{1 + 1/(d+2)},
\end{align}by using
\begin{align*}
\Vert (\theta-h_{n+1})_+ \Vert_{\LP{\infty}(Q_{r_{n+1}})} &= \sup_{A(h_{n+1},r_{n+1})} \theta - h_{n+1} \notag\\
& \leq \sup_{Q_{r_{n+1}}} \blu{\theta} - h_{n+1} \leq \sup_{Q_{r_{0}}} \blu{\theta} - h_0 = M(r_0)-h_0
\end{align*}
which holds since $M(r_n) \leq M(r_0)$ and $h_n \geq h_0$. Let $B = 2^{4 + 2 (d+2)}$. We choose $h_\infty$ large enough so that
\begin{align}\label{eq:hinfcond}
  \frac{C (M(r_0) - h_0)^{2/(d+2)}}{(h_\infty-h_0)^{4/(d+2)}r_0}\; b_0^{1/(d+2)}\leq \frac{1}{B},
\end{align}then by induction we obtain from \eqref{eq:13} that $b_n \leq b_0/B^n$, and therefore $b_n\rightarrow 0$ as $n\rightarrow \infty$. This implies that $\sup_{Q_{r_0/2}} \theta \leq h_\infty$. A simple calculation shows that if we let
\begin{align}\label{eq:hinfchoice}
  h_\infty = h_0 + \frac{C B^{(d+2)/4} (M(r_0)-h_0)^{1/2} b_0^{1/4}}{r_0^{(d+2)/4}}
\end{align}then \eqref{eq:hinfcond} holds. Lastly, $b_0=b(h_0,3r_0/4)$ may be bounded via \eqref{eq:10} and the H\"older inequality as
\begin{align}\label{eq:b0bound}
b_0 \leq C a(h_0,r_0)^{1/(d+2)} r_0^{d+2} (M(r_0)- h_0)^2.
\end{align}The proof of the lemma is concluded by combining $\sup_{Q_{r_0/2}} \theta \leq h_\infty$ with \eqref{eq:hinfchoice} and \eqref{eq:b0bound}. \blu{From the above proof it follows that inequality \eqref{eq:sup:halfball} also holds with $\theta$ is replaced by $-\theta$.}
\end{proof}

As opposed to the elliptic case, in the parabolic theory we need an additional energy inequality to control the possible growth of level sets of the solution.
\begin{lemma}[\bf Second energy inequality]\label{lemma:2nd:energyineq}
Let $\theta \in \LL{\infty}{2} \cap \LH{2}$ be a global weak solution of the initial value problem associated to \eqref{eq:theta:1}--\eqref{eq:u:Vdef}. Furthermore, assume that $V_{ij} \in \LBMO{\infty}$ for all $i,j\in\{1,\ldots,d\}$, and \eqref{eq:V:integrability} holds. Fix an arbitrary $x_0 \in \RRd$, let $h\in{\mathbb R}$, $0<r<R$, and $0<t_1 < t_2$. Then we have
\begin{align}\label{eq:2nd:energyineq}
&\Vert (\theta(t_2,\cdot)-h)_+ \Vert_{L^2(B_r)}^2\notag\\
& \qquad \leq \Vert (\theta(t_1,\cdot)-h)_+ \Vert_{L^2(B_R)}^2 + \frac{C\, R^{d} (t_2 - t_1)}{(R-r)^2}  \Vert (\theta - h)_+ \Vert_{\LP{\infty}( (t_1,t_2)\times B_R)}^2
\end{align}for some sufficiently large positive constant $C = C(d,\Vert V_{ij} \Vert_{\LBMO{\infty}})$, where we have denoted $B_\rho = B_\rho(x_0)$ for $\rho > 0$.
\end{lemma}
\begin{proof}Note that by Lemma~\ref{lemma:L^infty} we have that $\theta \in \LP{\infty}$ and hence the right side of \eqref{eq:2nd:energyineq} is finite. Let $\eta\in C_{0}^{\infty}(\RRd)$ be a smooth cutoff such that $\eta\equiv 1$ on $B_r$, $\eta \equiv 0$ on $B_R^c$, and $|\nabla \eta(x)| \leq C/(R-r)$, for all $x\in \RRd$, for some constant $C>0$. Multiply \eqref{eq:theta:1} by $\eta^2 (\theta-h)_+$ and integrate from $t_1$ to $t_2$ to obtain
\begin{align*}
  \intinttt{t_1}{t_2}{\RRd} \partial_t \Big( (\theta-h)_+\Big)^2 \eta^2 dxdt &- 2 \intinttt{t_1}{t_2}{\RRd} \partial_{jj}(\theta-h)_+ (\theta-h)_+ \eta^2 dxdt\notag\\
  &\qquad = - \intinttt{t_1}{t_2}{\RRd} \partial_{i} V_{ij}\; \partial_j \Big( (\theta-h)_+\Big)^2 \eta^2 dx dt\notag\\
  &\qquad = - \intinttt{t_1}{t_2}{\RRd} \partial_{i} \tilde{V}_{ij,R}\; \partial_j \Big( (\theta-h)_+\Big)^2 \eta^2 dx dt,
\end{align*}
where, as in \eqref{eq:V:mean:def}, we have denoted $ \tilde{V}_{ij,R}(t,x) = V_{ij}(t,x) - \frac{1}{|B_R|} \int_{B_R} V_{ij}(t,y)\; dy$. After integrating by parts we get
\begin{align*}
  &\int_{\RRd} (\theta(t_2,\cdot)-h)_+^2 \eta^2 dx + \intinttt{t_1}{t_2}{\RRd} |\nabla (\theta-h)_+|^2 \eta^2 dx dt\notag\\
  & \qquad  \leq \int_{\RRd} (\theta(t_1,\cdot)-h)_+^2 \eta^2 dx\notag\\
  & \qquad \ \ \  + C \intinttt{t_1}{t_2}{\RRd} (\theta-h)_+^2 \left( |\partial_j(\eta \partial_j \eta)| + |\tilde{V}_{ij,R}| |\partial_i (\eta \partial_j \eta)| + |\tilde{V}_{ij,R}|^2 |\partial_i \eta \partial_j \eta|\right) dx dt.
\end{align*}We bound the right side of the above estimate as in \eqref{eq:1stBMObound} and \eqref{eq:2ndBMObound} to obtain that
\begin{align*}
& \Vert (\theta(t_2,\cdot)-h)_+ \Vert_{L^2(B_r)}^2\notag\\
&\qquad \leq \Vert (\theta(t_1,\cdot)-h)_+ \Vert_{L^2(B_R)}^2 + \frac{C R^{d\varepsilon/2} (t_2-t_1)^{\varepsilon/2}}{(R-r)^2} \Vert  (\theta-h)_+ \Vert_{\LP{2}}^{2-\varepsilon} \Vert (\theta-h)_+ \Vert_{\LP{\infty}}^\varepsilon
\end{align*}Letting $\varepsilon = 2$ in the above estimates concludes the proof of the lemma. \blu{The corresponding statement for $-\theta$ also holds.}
\end{proof}
The use of the second energy inequality is to bound $|\{\theta(t_2,\cdot)\geq H\}\cap B_R|/|B_R|$, whenever $|\{\theta(t_1,\cdot)\geq h\}\cap B_r|/|B_r|\leq 1/2$. More precisely, we have the following lemma.
\begin{lemma}\label{lemma:controloflevelset}Fix $\kappa_0 = (4/5)^{1/d}$, let $n_0\geq 2$ be the least integer so that $2^{n_0}/(2^{n_0}-2) \leq \sqrt{6/5}$, and let $\delta_0 = (1-\kappa_0)^2/(12 C_0 \kappa_0^2)$, where $C_0$ is the constant from \eqref{eq:2nd:energyineq}. For $t_1,R>0$, if
\begin{align}\label{eq:cond}
|\{ \theta(t_1,\cdot) \geq h\} \cap B_r | \leq \frac 12 |B_r|,
\end{align}then for all $t_2 \in [t_1,t_1 + \delta_0 r^2]$ we have
\begin{align}\label{eq:cond}
|\{ \theta(t_2,\cdot) \geq H\} \cap B_R | \leq \frac 78 |B_R|,
\end{align}where $r=\kappa_0 R$, $M=\sup_{(t_1,t_1+\delta_0 R^2)\times B_R} \blu{\theta}$, $m=\inf_{(t_1,t_1+\delta_0 R^2)\times B_R} \blu{\theta}$, $h = (M+m)/2$, and $H = M-(M-m)/2^{n_0}$.
\end{lemma}
\begin{proof}
For $t_2 \in [t_1,t_1 + \delta_0 r^2]$, we obtain from the second energy inequality (cf.~\eqref{eq:2nd:energyineq}) that
\begin{align}\label{eq:19}
\Vert (\theta(t_2,\cdot)-h)_+ \Vert_{L^2(B_r)}^2 &\leq \Vert (\theta(t_1,\cdot)-h)_+ \Vert_{L^2(B_R)}^2 + \frac{C_0 R^d (t_2-t_1)}{(R-r)^2} \Vert (\theta-h)_+ \Vert_{\LP{\infty}(Q_2)}^2\notag\\
& \leq \Vert (\theta(t_1,\cdot)-h)_+ \Vert_{L^2(B_R)}^2 + \frac{C_0 R^d \delta_0 r^2}{(R-r)^2} (M-h)^2,
\end{align}where $Q_2 = (t_1,t_1+\delta_0 R^2)\times B_R$. The left side of the above estimate is bounded from below as
\begin{align}\label{eq:20}
\Vert (\theta(t_2,\cdot)-h)_+ \Vert_{L^2(B_r)}^2  &\geq  \Vert (\theta(t_2,\cdot)-h)_+ \Vert_{L^2(B_r \cap \{ \theta(t_2,\cdot) \geq H\})}^2\notag\\
&\geq (H-h)^2 | \{ \theta(t_2,\cdot) \geq H\} \cap B_r|.
\end{align}From \eqref{eq:19}, \eqref{eq:20}, and the H\"older inequality, we obtain after dividing by $|B_r|$ that
\begin{align*}
  \frac{| \{ \theta(t_2,\cdot) \geq H\} \cap B_r|}{|B_r|} \leq  \frac{(M-h)^2 R^d}{(H-h)^2 r^d} \left( \frac{| \{ \theta(t_1,\cdot) \geq h\} \cap B_R|}{|B_R|} + \frac{C_0  \delta_0 r^2}{(1-r/R)^2 R^2}\right).
\end{align*}Noting that by construction $(M-h)/(H-h) = 2^{n_0}/(2^{n_0}-2) \leq \sqrt{6/5}$, and recalling that $r/R = \kappa_0 = (4/5)^{1/d}$, we obtain from the previous estimate and the assumption of the lemma that
\begin{align}\label{eq:21}
\frac{|\{ \theta(t_2,\cdot) \geq H\} \cap B_r|}{|B_r|} & \leq \frac 32 \left( \frac{| \{ \theta(t_1,\cdot) \geq h\} \cap B_R|}{|B_R|} + \frac{C_0 \delta_0 \kappa_0^2}{(1-\kappa_0^2)}\right)\notag\\
& \leq \frac 32 \left( \frac 12 + \frac{1}{12}\right) = \frac{7}{8},
\end{align}concluding the proof of the lemma.
\end{proof}
\subsection{H\"older regularity of the solution}
We now have all necessary ingredients to conclude the De Giorgi argument \red{for} proving H\"older regularity of the weak solution.

Recall that since $\ddiv v = 0$, by Lemma~\ref{lemma:L^infty} we have that $\theta\in L^\infty([t_0,\infty);L^\infty(\RRd))$ for any $t_0 >0$. Moreover, if $V_{ij} \in L^\infty([t_0,\infty);BMO(\RRd))$ for some $t_0>0$, we obtain the energy inequalities of Lemmas~\ref{lemma:1st:energyineq} and \ref{lemma:2nd:energyineq}. In turn, these inequalities give control for the growth of the supremum on doubling cylinders (cf.~Lemma~\ref{lemma:sup:halfball}), and for the growth of level sets of the solution (cf.~Lemma~\ref{lemma:controloflevelset}). The rest of the proof follows as in~\cite{Lieberman}, but we give a sketch for the sake of completeness.

\begin{proof}[Proof of Theorem~\ref{thm:Holder:abstract}]
  The proof of the theorem is based on showing that there exists $\gamma \in (0,1)$ such that $\osc(Q_1) \leq \gamma\; \osc(Q_2)$. The key observation is that if $\gamma$ is independent of $R$, this estimate implies the H\"older regularity of the solution, where the H\"older exponent $\alpha\in(0,1)$ may be calculated explicitly from $\gamma$.

  Fix $\kappa_0,\delta_0,n_0,M,m,h,H,r$, and $R$ as in Lemma~\ref{lemma:controloflevelset} for the rest of this proof. We also fix two cylinders $Q_1 = [t_1,t_1 + \delta_0 r^2] \times B_r$, and $Q_2 = [t_1,t_1 + \delta_0 R^2]\times B_R$, where we recall that $t_1>0$ and $R>0$ are arbitrary.

  \blu{Recall that $h=(\inf_{Q_2} \theta + \sup_{Q_2} \theta)/2$. Without loss of generality we may assume $|\{ \theta(t_1,\cdot) \geq h\} \cap B_r| \leq |B_r| /2$. Otherwise , letting $h' =(\inf_{Q_2} (-\theta) + \sup_{Q_2} (-\theta))/2$ we have $|\{ -\theta(t_1,\cdot) \geq h'\} \cap B_r| = |\{ \theta(t_1,\cdot) \leq h\} \cap B_r| \leq |B_r| /2$, and we work with $-\theta$ instead of $\theta$.}

  For $n \geq n_0$, we define $H_n = M - (M-m)/2^n$, and note that $H = H_{n_0} \leq H_n \nearrow M$. We also let $w$ be $\theta$ truncated between levels $H_{n-1}$ and $H_n$, namely
\[w = \min\{ \theta,H_n\} - \min\{\theta,H_{n-1}\} =
\left\{
\begin{array}{ll}
0, & \theta < H_{n-1} \\
\theta-H_{n-1}, & H_{n-1}\leq \theta < H_n\\
H_n - H_{n-1}, & H_n \leq \theta.
\end{array}
\right.\]
Since $|\{ \theta(t_1,\cdot) \geq h\} \cap B_r| \leq |B_r| /2$, by Lemma~\ref{lemma:controloflevelset}, for every $t\in [t_1,t_1 + \delta_0 r^2]$ we have
\begin{align*}
  |\{ w(t,\cdot) = 0\} \cap B_R | &= |\{ \theta(t,\cdot) < H_{n-1}\} \cap B_R | \geq |\{ \theta(t,\cdot) < H\} \cap B_R | \geq \frac 78 |B_R|.
\end{align*}By the above estimate and the Poincar\'e inequality we obtain
\begin{align*}
  \int_{B_r} |w(t,\cdot)| dx \leq C r \int_{B_r} |\nabla w(t,\cdot)| dx
\end{align*}for all $t\in [t_1,t_1 + \delta_0 r^2]$, where $C = C(d)$ is a universal positive constant. Integrating the above estimate in time over $[t_1,t_1 + \delta_0 r^2]$ and using the H\"older inequality we get
\begin{align}
  \intint{Q_1} |w| dx dt &\leq C r \intint{Q_1} |\nabla w| dx dt\notag\\
& \leq C r |\{ H_{n-1} \leq \theta < H_n\} \cap Q_1|^{1/2} \Vert \nabla(\theta-H_{n-1})_+ \Vert_{\LP{2}(Q_1)}.\label{eq:24'}
\end{align}We bound the far right side of \eqref{eq:24'} by using Corollary~\ref{corr:energy}, to obtain
\begin{align}
  \intint{Q_1} |w| dx dt &\leq C r |\{ H_{n-1} \leq \theta < H_n\} \cap Q_1|^{1/2} \Vert \nabla(\theta-H_{n-1})_+ \Vert_{\LP{\infty}(Q_2)} \frac{|Q_2|^{1/2}}{R-r}\notag\\
& \leq C \frac{\kappa_0}{1-\kappa_0} |\{ H_{n-1} \leq \theta < H_n\} \cap Q_1|^{1/2} |Q_2|^{1/2} (M-H_{n-1})\label{eq:24''}
\end{align}The left side of \eqref{eq:24''} is bounded from below as
\begin{align}
  \intint{Q_1} |w| dx dt \geq \intint{Q_1\cap \{\theta \geq H_n\}} |w| dx dt \geq (H_n - H_{n-1}) |\{\theta\geq H_n\} \cap Q_1|.\label{eq:24'''}
\end{align}By combining and squaring estimates \eqref{eq:24''} and \eqref{eq:24'''} we obtain
\begin{align}
  |\{ \theta \geq H_n \} \cap Q_1 |^2 &\leq \frac{C |Q_2| (M-H_{n-1})^2}{(H_n - H_{n-1})^2} |\{ H_{n-1} \leq \theta < H_n\} \cap Q_1|\notag\\
& \leq C |Q_2| \left( |\{ \theta \geq H_{n-1} \} \cap Q_1 | - |\{ \theta \geq H_n \} \cap Q_1 | \right),\label{eq:24}
\end{align}where we used the fact that, by construction, $(M-H_{n-1})/(H_n - H_{n-1}) = 2$. Hence,
\begin{align*}
  \sum\limits_{n\geq n_0+1}|\{ \theta \geq H_n \} \cap Q_1 |^2 \leq C |Q_2| |\{ \theta \geq H_{n_0} \} \cap Q_1 |,
\end{align*}and since the sequence $|\{ \theta \geq H_n \} \cap Q_1 |$ is decreasing, we obtain
\begin{align*}
  |\{ \theta \geq H_n \} \cap Q_1 | \leq \frac{C |Q_2|^{1/2} |\{\theta\geq H\} \cap Q_1|^{1/2}}{(n-n_0)^{1/2}}
\end{align*}for all $n\geq n_0+1$. By Lemma~\ref{lemma:controloflevelset} we have that $|\{\theta\geq H\} \cap Q_1| \leq 7|Q_1|/8$, and therefore the above estimate implies
\begin{align}\label{eq:25}
  |\{ \theta \geq H_n \} \cap Q_1 | \leq \frac{C r^{d+2}}{(n-n_0)^{1/2}},
\end{align}where we have used that $r = \kappa_0 R$, and $\kappa_0 = \kappa_0(d)$. By Lemma~\ref{lemma:sup:halfball}, the fact that $\delta_0 <1$, and the estimate \eqref{eq:25} we obtain
\begin{align*}
  \sup_{Q_1} \blu{\theta} &\leq H_n + C \left( \frac{|\{ \theta \geq H_n \} \cap Q_1 |^{1/(d+2)}}{r}\right)^{1/2} (M-H_n)\notag\\
& \leq H_n + \frac{C}{(n-n_0)^{1/(4 d+8)}} (M-H_n),
\end{align*}for some positive constant $C = C(d,\Vert V_{ij} \Vert_{\LBMO{\infty}})$, which is independent of $r$. Therefore there exists a sufficiently large $n_1 =n_1(d,\Vert V_{ij} \Vert_{\LBMO{\infty}})\geq n_0+1$ such that
\begin{align*}
   \sup_{Q_1} \blu{\theta} &\leq H_{n_1} + \frac 12 (M-H_{n_1}).
\end{align*}Recalling the definition of $H_n,m$, and $M$, a simple calculation shows that the above estimate implies
\begin{align}
  \osc(Q_1) = \sup_{Q_1} \blu{\theta} - \inf_{Q_1} \blu{\theta} &\leq H_{n_1} - m + \frac 12 (M-H_{n_1}) = \left(1 - \frac{1}{2^{n_1+2}}\right) (M-m)\notag\\
& = \left(1 - \frac{1}{2^{n_1+2}}\right) \left( \sup_{Q_2} \blu{\theta} - \inf_{Q_2} \blu{\theta}\right) = \gamma\; \osc(Q_2),\label{eq:Holder}
\end{align}where $\gamma = 1 - 1/2^{n_1+2} \in (0,1)$ is independent of $r$. Recall that in \eqref{eq:Holder} we have $Q_1 = [t_1,t_1 + \delta_0 \kappa_0^2 R^2] \times B_{\kappa_0 R}(x_0)$ and $Q_2 = [t_1,t_1 + \delta_0 R^2] \times B_{R}(x_0)$, with $\kappa_0,\delta_0$ fixed positive constants, and $R>0$ arbitrary. This classically implies H\"older continuity of $\theta$ at the arbitrary point $(t_1,x_0)\in (0,\infty)\times \RRd$, concluding the proof of the theorem.
\end{proof}

\section{Global regularity for a nonlinear parabolic equation}\label{sec:nonlinear}
We address the global regularity of solutions to the initial value problem
\begin{align}
  & \partial_t \theta - \Delta \theta + (u\cdot\nabla) \theta = 0 \label{eq:N1}\\
  & \ddiv u = 0\\
  & u_j = \partial_i T_{ij} \theta\\
  & \theta(0,\cdot) = \theta_0,\label{eq:N4}
\end{align}where $\{T_{ij}\}_{i,j=1}^{d}$ is a matrix of Calder\'on-Zygmund singular integral operators such that $\partial_i \partial_j T_{ij} f = 0$ for any Schwartz function $f$. As an elementary example, if $d=2$ we may consider $T_{11}=T_{22} = 0$, and $T_{12}=-T_{21} = T$, for some Calder\'on-Zygmund operator $T$ (for instance $T=R_{i}$, a Riesz-Transform). In this case the velocity would be $u = \nabla^\perp T\theta$. When $d=3$, a physical example of such a matrix $\{T_{ij}\}$ \blu{arises in the MG system (cf.~Section~\ref{sec:MG} below)}.
\begin{theorem}[\bf The nonlinear problem]\label{thm:nonlinear}
Let $\theta_0\in L^2(\RRd)$ be given. A Leray-Hopf weak solution $\theta \in L^\infty([0,\infty);L^2(\RRd)) \cap L^2((0,\infty);H^1(\RRd))$ of \eqref{eq:N1}--\eqref{eq:N4}, evolving from $\theta_0$,
is a classical solution, that is $\theta \in C^\infty((0,\infty)\times \RRd)$.
\end{theorem}
\begin{lemma}[\bf Boundedness]\label{lemma:bounded2}
A Leray-Hopf weak solution $\theta$ of \eqref{eq:N1}-\eqref{eq:N4} is bounded for $t>0$, i.e., $\theta \in L^\infty([t_0,\infty);L^\infty(\RRd))$ for any $t_0>0$.
\end{lemma}
\begin{proof}
  The proof of this lemma is the same as the proof of Lemma~\ref{lemma:L^infty} (cf.~\cite{CaffVass,ConstWu09}), and only uses the fact that $\ddiv u=0$, where $u\in \LP{2}((0,\infty)\times\RRd)$.
\end{proof}
Since $\theta \in \LP{\infty}$, it follows from the Calder\'on-Zygmund theory of singular integrals that $T_{ij}  \theta =: V_{ij} \in L^\infty([t_0,\infty);BMO(\RRd))$, for any $t_0>0$, where $i,j\in\{1,\ldots,d\}$. Therefore, we may treat \eqref{eq:N1} as a linear evolution equation (see also~\cite{CaffVass,ConstWu09}), where the divergence-free velocity field $u$ is given, and $u\in L^2((0,\infty);L^2(\RRd)) \cap L^\infty([t_0,\infty);BMO^{-1}(\RRd))$, for any $t_0>0$. This is precisely the setting of Theorem~\ref{thm:Holder:abstract} for the linear evolution equation. \red{Hence Theorem~\ref{thm:Holder:abstract} can be applied to the nonlinear problem to give H\"older regularity of the solution.} Therefore we obtain:
\begin{lemma}[\bf H\"older regularity]\label{lemma:Holder2}
A Leray-Hopf weak solution $\theta$ of \eqref{eq:N1}-\eqref{eq:N4} is H\"older smooth for positive time, i.e., for any $t_0>0$, there exists $\alpha>0$ such that $\theta \in C^\alpha([t_0,\infty)\times \RRd)$.
\end{lemma}
Lastly, since the H\"older regularity is sub-critical for the natural scaling of \eqref{eq:N1}-\eqref{eq:N4} one may bootstrap to prove that the solution is in a higher H\"older class:
\begin{lemma}[\bf Higher regularity]\label{lemma:C^infty}
Let $\theta \in L^\infty([t_0,\infty);C^\alpha(\RRd))$ be a Leray-Hopf weak solution of the initial value problem associated to \eqref{eq:theta:1}--\eqref{eq:u:Vdef}, \chd{with $\alpha\in(0,1)$}.
Then $\theta \in L^\infty([t_1,\infty);\chd{C^{1+\delta}}(\RRd))$, for any $t_1 > t_0$, for some \chd{$\delta \in(0,1)$}.
\end{lemma}
\chd{For $1/2< \alpha<1$, the} proof is the same as the proof of higher regularity for the modified surface quasi-geostrophic equation~\cite[Theorem 2.2]{ConstIyerWu} (see also~\cite[Theorem 3.1]{ConstWu08} for the supercritical quasi-geostrophic equation). These elegant proofs use the natural characterization of H\"older spaces in terms of Besov spaces, and energy inequalities at the level of frequency shells.

\chd{For $0<\alpha\leq  1/2$, the $C^\alpha$ smoothness of $\theta$ is weak relative to the roughness of the velocity $u$, and it is therefore necessary to modify the techniques of \cite{ConstWu08,ConstIyerWu} for the proof of higher regularity. In \cite{FV-higher} we give the details of this modification which uses the extra information that $u\in L_{t,x}^2$ and employs estimates in the Chemin-Lerner (cf.~\cite{CheminLerner}) space-time Besov spaces.}

\chd{We give a very brief outline of the proof of lemma 3.4
in the two ranges for alpha and refer the reader to \cite{ConstWu08,ConstIyerWu} and
\cite{FV-higher} for detailed estimates.}

\begin{proof}[Proof of Lemma~\ref{lemma:C^infty}]
    \chd{Let $\dot{B}_{p,q}^s$ be the classical homogenous Besov space (cf.~\cite{ConstIyerWu,ConstWu08}), and recall that $L^\infty \cap \BB{s}{\infty} = C^s$ is the H\"older space with index $s$.} The proof of the lemma \chd{in the case $\alpha \in (1/2,1)$} is based on first noting that if $\theta$ is as in the statement of the lemma, then $\theta \in L^{\infty}([t_0,\infty); \BB{\alpha_p}{p})$, where $\alpha_p = (1-2/p)\alpha$, and $p\in[2,\infty)$ is fixed, to be chosen later. Then, for $j\in {\mathbb Z}$ fixed, we have
  \begin{align}\label{eq:CIW1}
    \frac{1}{p}\frac{d}{dt} \Vert \Delta_j \theta \Vert_{L^p}^p +\int |\Delta_j \theta|^{p-2} \Delta_j \theta (-\Delta) \Delta_j \theta  = - \int |\Delta_j \theta|^{p-2} \Delta_j \theta \Delta_j (u\cdot \nabla \theta).
  \end{align}Upon integration by parts (see also~\cite{CannonePlanchon}), the dissipative term is bounded from below
  \begin{align}\label{eq:CIW2}
     \int |\Delta_j \theta|^{p-2} \Delta_j \theta (-\Delta) \Delta_j \theta\; dx \geq \frac{2^{2j}}{C(d,p)} \Vert \Delta_j \theta \Vert_{L^p}^p,
  \end{align}where $C(d,p)>0$ is a constant depending on the dimension and $p$. The main difficulty lies in estimating the convection term. This is achieved in \cite{ConstIyerWu,ConstWu08} by using the Bony paraproduct formula, the H\"older inequality, the Bernstein inequalities, a commutator estimate, and the fact that $\Vert u \Vert_{C^{\alpha_p-1}}\leq C \Vert \theta \Vert_{C^{\alpha_p}}$. The latter holds since $u_j = \partial_i T_{ij} \theta$ and the fact that Calder\'on-Zygmund operators are bounded on H\"older spaces. If $\alpha_p <2$ these \chd{operations give}
  \begin{align}\label{eq:CIW3}
    \left| \int |\Delta_j \theta|^{p-2} \Delta_j \theta \Delta_j (u\cdot \nabla \theta)\; dx \right| \leq C 2^{(2-2\alpha_p )j} \Vert \theta \Vert_{C^{\alpha_p}} \Vert \theta \Vert_{\BB{\alpha_p}{p}}.
  \end{align}
\chd{Combining \eqref{eq:CIW1}--\eqref{eq:CIW3}, using the Gr\"onwall inequality, and then taking the supremum in $j$ gives that $\theta \in L^{\infty}([t_1,\infty); \BB{2\alpha_p}{p}(\RRd))$ for any $t_1>t_0$. Using the Besov embedding theorem we obtain that $\theta \in L^\infty([t_1,\infty);\BB{2\alpha - \epsilon_p}{\infty}(\RRd))$, for any $t_1>t_0$, where $\epsilon_p = (4\alpha+d)/p < (4+d)/p$. Letting $p> (4+d)/(2\alpha-1)$ concludes the proof of the lemma in the case $\alpha\in(1/2,1)$.}

\chd{In the case $\alpha \in(0,1/2]$ the proof is based on proving that the additional information $\theta \in L^2([t_1,t_2];\dot{H}^1)$, implies $\theta \in L^2([t_1,t_2]; \dot{B}^{1+d/p}_{p,1})$ for some large enough $p>2$, and for any $t_2 > t_1$. This is achieved by using the smoothing effect of the Laplacian on high frequencies of $\theta$, so that we need to work in the space-time Besov spaces introduced by Chemin and Lerner (cf.~\cite{CheminLerner}). By the endpoint Sobolev embedding theorem we thus obtain that $\nabla \theta \in L^2([t_1,t_2];B^{0}_{\infty,1}) \subset L^2([t_1,t_2];L^\infty)$. From here, standard energy estimates imply that $\theta\in L^\infty([t_1',t_2];\dot{H}^m)$ for all $m\geq 2$, and $t_1'\in[t_1,t_2]$, concluding the proof of the lemma after applying the Sobolev embedding $H^m \subset C^{1,\beta}$ with $m>1+d/2$. We refer to~\cite{FV-higher} for details.}

\end{proof}

\begin{proof}[Proof of Theorem~\ref{thm:nonlinear}]
The existence of a global in time Leray-Hopf weak solution of \eqref{eq:N1}--\eqref{eq:N4}, evolving from $\theta_0 \in L^2$, is proven in Appendix~\ref{appendix}. The argument is to construct solutions to an approximate system, and then to pass to the limit in the weak formulation of the problem, using the Aubin-Lions compactness lemma (cf.~\cite{LionsJL69}).

The proof of Theorem~\ref{thm:nonlinear} now follows from Lemmas~\ref{lemma:bounded2}, \ref{lemma:Holder2}, \ref{lemma:C^infty}. For any $\beta \in (0,1)$, after finitely many applications of Lemma~\ref{lemma:C^infty} the solution is shown to be in $L^\infty([t_0,\infty); C^{1+\beta}(\RRd))$, for any $t_0>0$, and is hence a classical solution. Higher regularity is standard.\end{proof}

\section{Global regularity of the MG system}\label{sec:MG}
There is a vast literature studying mathematical models for the Earth's dynamo (see, for example Glatzmaier, Ogden, and Clune~\cite{Glatzmaier} and references therein). However, at present, no computational dynamo model can encompass the fine scale resolution required to simulate the turbulent processes believed to exist in the Earth's core. It is therefore reasonable to examine models that are simpler than the full system of PDE governing rotating, convective, magneto-hydrodynamic flows, but that retain some of the essential features relevant to the physics of the Earth's core. One such model for magnetostrophic turbulence was recently proposed by Moffatt~\cite{Moffatt}. He postulates that the magnetic field $B(t,x)$ in the core consists of a mean part $B_0$, which results from dynamo action and can be considered as locally uniform and steady, and a perturbation field $b(t,x)$ induced by the flow $u(t,x)$ across $B_0$.

It is assumed that the scale $L$ of convective turbulence lies in the range $V/\Omega \ll L \ll \eta/V$, where $V$ is the average magnitude of the upward buoyant velocity, $\Omega$ is the angular velocity of the Earth, and $\eta$ is the magnetic diffusivity of the fluid medium. This assumption implies that the Rossby number $V/\Omega L$ and the magnetic Reynolds number $VL/\eta$ are both small. The turbulent Reynolds number in the core is expected to be very large. The dominant terms in the \red{three dimensional} equations of motion and the induction equation give the following {\it linear} system
\begin{align}
  &2 \Omega e_3 \times u = -\nabla P + (B_0 \cdot \nabla) b - \theta g \label{eq:MG:2.1}\\
  & 0 = (B_0 \cdot \nabla) u + \eta \Delta b \label{eq:MG:2.1'}\\
  & \ddiv u = 0 \label{eq:MG:2.3}\\
  & \ddiv b = 0 \label{eq:MG:2.4},
\end{align}where $P(t,x)$ is the sum of the fluid and magnetic pressures, $\theta(t,x)$ is the buoyancy field (e.g. perturbation of the temperature), and $g$ is the gravitational acceleration. \red{We use Cartesian coordinates in the reference frame rotating about the axis $e_3=(0,0,1)$.}

Equations \eqref{eq:MG:2.1}-\eqref{eq:MG:2.4} establish a linear relation between the variables $u(t,x)$, $b(t,x)$, and $\theta(t,x)$. The sole remaining nonlinearity from the full convective MHD system occurs in the advection-diffusion equation for the buoyancy $\theta(t,x)$:
\begin{align}\label{eq:MG:2.5}
\partial_t \theta + (u \cdot \nabla) \theta = S+ \kappa \Delta \theta,
\end{align}where $S$ is a source term. The diffusivity $\kappa$ in the core is very small, hence the nonlinear advection term is dominant and cannot be neglected.

The system \eqref{eq:MG:2.1}-\eqref{eq:MG:2.5} gives an active scalar model for magneto-geostrophic dynamics, which we call the {\it \MG equations}. As Moffatt observes, \eqref{eq:MG:2.1}-\eqref{eq:MG:2.5} has some similarities with the dissipative Burgers equation, but it has a clearer physical basis and the velocity $u(t,x)$ is three-dimensional. We remark that the system has closer similarities to the surface quasi-geostrophic equation (SQG), which is also derived in the context of a rapidly rotating system dominated by Coriolis' force. However, the operator that connects $u$ and $\theta$ via \eqref{eq:MG:2.1}-\eqref{eq:MG:2.4} has features that are distinct from the analogous operator in the SQG system as we shall now discuss.

For simplicity we will examine \eqref{eq:MG:2.1}-\eqref{eq:MG:2.4} in the case where $B_0$ is a vector that is constant in magnitude and direction in the plane perpendicular to $e_3$. We write
\begin{align*}
  B_0  = \beta e_2.
\end{align*}We assume that gravity acts parallel to the axis of rotation, i.e. $g = e_3$. With these assumptions we are examining a local tangent plane model for the Earth's fluid core that ignores the sphericity, but retains the essence of the mathematical structure of the active scalar equation \eqref{eq:MG:2.5}, with $u$ constructed from $\theta$ via \eqref{eq:MG:2.1}-\eqref{eq:MG:2.4}. Manipulation of the linear system \eqref{eq:MG:2.1}-\eqref{eq:MG:2.4} gives, in component form,
\begin{align}
  & u_1 = D^{-1} \left( - 2\Omega \partial_2 P - \Gamma \partial_1 P\right) \label{eq:MG:2.6}\\
  & u_2 = D^{-1} \left( 2\Omega \partial_1 P - \Gamma \partial_2 P\right) \label{eq:MG:2.7}\\
  & \partial_3 u_3 = D^{-1} \Gamma \Delta_H P \label{eq:MG:2.8}\\
  & \partial_3 \theta = \left( \Gamma^2 \Delta_H D^{-1} + \partial_{33}\right) P \label{eq:MG:2.9},
\end{align}where the operators $\Gamma,D$, and $\Delta_H$ are defined as
\begin{align}
  &\Gamma = - \frac{\chd{\beta^2}}{\eta} (-\Delta)^{-1} \partial_{22}\label{eq:MG:2.10}\\
  &D = 4 \Omega^2 + \Gamma^2 \label{eq:MG:2.11},\\
  &\Delta_H = \partial_{11} + \partial_{22},
\end{align}where $x = (x_1,x_2,x_3) \in \xdomain$. We note that a more general choice of the mean, steady, locally uniform magnetic field $B_0$ or of the gravitational vector $g$ results in the same structure of the leading order terms. It is the \chd{anisotropy} that is produced by $B_0$ that is a distinctive and crucial feature of the \MG system.

The operator $D$ given by \eqref{eq:MG:2.11} is invertible since its Fourier symbol does not vanish on $\Fdomain$, justifying the use of $D^{-1}$. In order to uniquely determine $u_3$ and $\theta$ from \eqref{eq:MG:2.8} and \eqref{eq:MG:2.9}, we restrict the system to the function spaces where $\theta$ and $u_3$ are periodic in the $x_3$-variable, with zero vertical mean, i.e.~$\int_{0}^{2\pi} \theta\; dx_3= \int_{0}^{2\pi} u_3 \; dx_3 =0$. In fact, without such a restriction the system is not well defined. We integrate \eqref{eq:MG:2.9} and use the zero-mean assumption to obtain
\begin{align}\label{eq:MG:2.12}
\theta = A[P].
\end{align}where $A$ is formally defined as the Fourier multiplier with symbol
\begin{align}
  \hat{A}(\kkx,\kky,\kkz) = \frac{4\Omega^2 \kkz^{2} |\KK|^2 + (\chd{\beta^2}/\eta)^2 \kky^{4}}{i \kkz ( 4 \Omega^2 |\KK|^4 + (\chd{\beta^2}/\eta)^2 \kky^{4})} \label{eq:MG:2.13}
\end{align}
for all $\kkz \neq 0$ (by our vertical mean-free assumption), where $\KK = (\kkx,\kky,\kkz)\in \Fdomain$. Therefore $A$ is invertible on the space of functions with null $x_3$-average. Note that $\partial_3 A[P] = (\Gamma^2 \Delta_H D^{-1} + \partial_{33}) P$ in the physical space. We now use \eqref{eq:MG:2.6}-\eqref{eq:MG:2.8} to represent $u_1,u_2$, and $u_3$ in terms of $\theta$:
\begin{align}
  &u_1 = D^{-1} ( - 2 \Omega \partial_2 - \Gamma \partial_1) (A^{-1}[\theta]) \equiv M_1[\theta] \label{eq:MG:2.14}\\
  &u_2 = D^{-1}(2 \Omega \chd{\partial_1} - \Gamma \partial_2) ( A^{-1}[\theta]) \equiv M_2[\theta] \label{eq:MG:2.15}\\
  &u_3 = (D^{-1} \Gamma \Delta_H) (D^{-1}\Gamma \Delta_H + \partial_{33})^{-1}[\theta] \equiv M_3[\theta] \label{eq:MG:2.16}.
\end{align}To investigate the properties of the operator $M=(M_1,M_2,M_3)$, we note that it is a vector of Fourier multipliers, with explicit Fourier symbols given by
\begin{align}
  & \hat{M}_1(\KK) = \frac{2\Omega \kky \kkz |\KK|^2 - (\chd{\beta^2}/\eta) \kkx \kky^{2} \kkz}{4 \Omega^2 \kkz^{2} |\KK|^2 + (\chd{\beta^2}/\eta)^2 \kky^{4}}\label{eq:MG:2.17}\\
  & \hat{M}_2(\KK) = \frac{-2\Omega \kkx \kkz |\KK|^2 - (\chd{\beta^2}/\eta) \kky^{3}\kkz}{4 \Omega^2 \kkz^{2} |\KK|^2 + (\chd{\beta^2}/\eta)^2 \kky^{4}}\label{eq:MG:2.18}\\
  & \hat{M}_3(\KK) = \frac{(\chd{\beta^2}/\eta) \kky^{2}(\kkx^{2} + \kky^{2})}{4 \Omega^2 \kkz^{2} |\KK|^2 + (\chd{\beta^2}/\eta)^2 \kky^{4}}\label{eq:MG:2.19}
\end{align}for all $k_3 \neq 0$. Since by assumption $\hat{\theta}(\kkx,\kky,0) = \hat{u}(\kkx,\kky,0)= 0$, in order to have a uniquely defined symbol $\hat{M}(k)$ on all of $\Fdomain$, without loss of generality we may let $\hat{M}_1(\kkx,\kky,0) = \hat{M}_2(\kkx,\kky,0)=0$, and $\hat{M}_3(\kkx,\kky,0) = \hat{M}_3(\kkx,\kky,1)$. Note that $u_j = M_j[\theta]$ is defined via the inverse Fourier transform from
\begin{align}
  &\hat{u}_j(\KK) = \hat{M}_{j}(\KK) \hat{\theta}(\KK),\ \mbox{for all}\ \KK\in \Fdomain, \label{eq:udef1}
\end{align}for all $j\in \{1,2,3\}$. Also, since $\ddiv u = 0$, we have that $\KK\cdot \hat{M}(\KK) = 0$.

When the frequency vector $\KK=(\kkx,\kky,\kkz)$ has components such that $\kkx \leq \max\{\kky,\kkz\}$, then the symbols $\hat{M}_j$ are bounded for all $j\in \{1,2,3\}$. However this is not the case for \chd{``curved''} regions of frequency space where $\kkz = O(1)$, $\kky = O(|\kkx|^{\chd{\sigma}})$, \chd{where $0\leq \sigma \leq 1/2$,} and $|\kkx|\gg 1$. In such regions the symbols \eqref{eq:MG:2.17}--\eqref{eq:MG:2.19} are unbounded, since as $|\kkx| \rightarrow \infty $ we have
\begin{align*}
  |\hat{M}_1(\kkx,\chd{|\kkx|^{\sigma}},1)| \approx \chd{|\kkx|^{\sigma}},\ |\hat{M}_2(\kkx,\chd{|\kkx|^{\sigma}},1)| \approx |\kkx|,\ | \hat{M}_3(\kkx,\chd{|\kkx|^{\sigma}},1)| \approx \chd{|\kkx|^{2 \sigma}},
\end{align*}\chd{where $\sigma \in (0,1/2]$,} and we write $a \approx b$ if there exists a constant $C>0$ such that $a/C \leq b\ \leq C a$. \red{It follows from \eqref{eq:MG:2.17}--\eqref{eq:MG:2.19} that}
\begin{align}
  |\hat{M}_j(\KK)| \leq C_* |\KK|\label{eq:MG:universalbound}
\end{align}for all $\KK \in \Fdomain$, and all $j\in \{1,2,3\}$, where $C_*=C_*(\beta,\eta,\Omega)>0$ is a fixed constant. From the previous remark it is clear that along certain \chd{curves} in frequency space the bound \eqref{eq:MG:universalbound} is sharp.

We now prove that the active scalar equation (cf.~\eqref{eq:MG:2.1}-\eqref{eq:MG:2.5} with $S=0$)
\begin{align}
  & \partial_t \theta + (u\cdot\nabla)\theta = \kappa \Delta\theta\label{eq:MG:2.20}\\
  & \ddiv u =0\\
  & u = M[\theta]
\end{align}with $M$ given by \eqref{eq:MG:2.14}-\eqref{eq:MG:2.16}, or equivalently by its Fourier symbol \eqref{eq:MG:2.17}-\eqref{eq:MG:2.19}, satisfies the conditions of the abstract problem studied in Section~\ref{sec:nonlinear}. First note that we can write
\begin{align}
  u_j = M_j[\theta] = \partial_i T_{ij}[\theta] = \partial_i V_{ij},
\end{align}where we have denoted
\begin{align}
T_{ij} = - \partial_i (-\Delta)^{-1} M_j.\label{eq:MG:2.27}
\end{align} By \eqref{eq:MG:universalbound} we have that $|\hat{T}_{ij}(\KK)|\leq C_*$ for all $\KK\in\Fdomain$, and \red{hence} it follows directly from Plancherel's theorem that $T_{ij} \colon L^2(\xdomain) \mapsto L^2(\xdomain)$ is a bounded map.

It remains to prove that $T_{ij} \colon L^\infty(\xdomain) \mapsto BMO(\xdomain)$ boundedly. This reduces to proving that $N_j = (-\Delta)^{-1/2} M_j\colon L^\infty \mapsto BMO$ is a bounded map, since Riesz-transforms are bounded on $BMO$. The later holds because $N_j$ is a pseudo-differential operator of order $0$ (cf.~\cite{McLean,Ruzhansky,Stein93}). The main idea is that one may extend $\hat{N}_j$ from $\Fdomain$ to a symbol $\hat{N}_j'$ defined on ${\mathbb R}^3$ such that they agree on $\Fdomain$, and such that $\hat{N}_j'$ is the symbol of a classical H\"ormander-class pseudo-differential operator of order $0$ (cf.~Stein~\cite{Stein93}). More precisely, let $\hat{N}_j'(\KK)=\hat{M}_j(\KK)/ |\KK|$ for all $\KK\in {\mathbb R}^3$ with $|\kkz|\geq 1$, while for $|\kkz|<1$, replace the denominator $4 \Omega^2 \kkz^{2} |\KK|^2 + (\chd{\beta^2}/\eta)^2 \kky^{4}$ by the quantity $4 \Omega^2 \varphi(\kkz)^2(\kkx^{2} + \kky^{2}+\varphi(\kkz)^2) + (\chd{\beta^2}/\eta)^2 \kky^{4}$, in the definitions \eqref{eq:MG:2.17}-\eqref{eq:MG:2.19} of $\hat{M}_j(\KK)$. Here $\varphi(\cdot)$ is $C^\infty$ smooth \blu{monotone increasing} function that coincides with the identity on $|\kkz|\geq 1/2$, and is constantly equal to $1/2$ on $|\kkz| \leq 1/4$. This construction ensures the smoothness \blu{of the symbol} near the origin, while the bound $|\partial_k^\alpha \hat{N}_j'(\KK)|\leq C_\alpha (1+|\KK|)^{-|\alpha|}$ follows by inspection. To close the argument, note that the operators $N_j$ and $N_j'$ differ by a compact operator in the symbol class $S^{-\infty}$(cf.~\cite{McLean,Ruzhansky} and references therein). This concludes the proof of the boundedness of $T_{ij} \colon L^\infty \mapsto BMO$.

\red{The abstract Theorem~\ref{thm:nonlinear} may therefore be applied to the MG equations in order to obtain the global smoothness of weak solutions, and hence we have proven:}
\begin{theorem}[\bf The MG system]\label{thm:nonlinearMG}
Let $\theta_0\in L^2(\RRd)$ be given. There exists a $C^\infty$ smooth classical solution $\theta(t,x)$, of \eqref{eq:MG:2.20}--\eqref{eq:MG:2.27}, evolving from $\theta_0$.
\end{theorem}

\appendix
\section{Existence of weak solutions to \eqref{eq:N1}--\eqref{eq:N4}}~\label{appendix}
Here we sketch the proof of existence of global Leray-Hopf weak solutions of \eqref{eq:N1}--\eqref{eq:N4} evolving from $\theta_0 \in L^2(\RRd)$. We follow the general strategy used to construct weak solutions of the Navier-Stokes equations (cf.~\cite{Temam01}). The main obstacle is the fact that $u$ is obtained from $\theta$ via a nonlocal operator of order $1$.

Denote by $(-\Delta)^{1/2} = \Lambda$ the square root of the Laplacian. Let $\phi \in C_{0}^{\infty}(\RRd)$ be positive, with $\int_{\RRd} \phi\; dx = 1$. Then $\phi_{\epsilon} = \epsilon^{-d} \phi(x/\epsilon)$, for $\epsilon>0$, is a standard family of mollifiers. We first consider the approximating system
\begin{align}
  &\partial_t \teps + (\ueps \cdot \nabla )\teps - \Delta \teps = - \epsilon \Lambda^3 \teps \label{eq:A1}\\
  & \ddiv \ueps = 0, u_j = \partial_i T_{ij} \teps \label{eq:A2}\\
  & \teps(0,\cdot) = \theta_{0}^{\epsilon}, \label{eq:A3}
\end{align}where $\theta_{0}^{\epsilon} = \phi_\epsilon \ast \theta_0$ represents the mollified initial data, and $T_{ij}$ are Calder\'on-Zygmund operators. Note that $\Vert \theta_{0}^{\epsilon} \Vert_{L^2} \leq \Vert \theta_0 \Vert_{L^2}$ for any $\epsilon >0$.

Let $s>d/2+1$ and fix $\epsilon >0$. Since $\Lambda^s \theta_{0}^{\epsilon} \in L^2(\RRd)$, and since $\epsilon\, \Lambda^3$ gives a sub-critical dissipation, from standard energy arguments it follows that
\begin{align*}
  \sup_{t\in[0,T]} \Vert \Lambda^s \teps(t) \Vert_{L^2} \leq C(\epsilon,d,\phi,T,\Vert \theta_0 \Vert_{L^2}),
\end{align*} where $C(\epsilon,d,\phi,T,\Vert \theta_0 \Vert_{L^2}) >0$ is a positive constant which is finite for any $T<\infty$. This a-priori estimate and a standard Galerkin approximation procedure ensures the global existence of a strong $H^s$ solution to \eqref{eq:A1}-\eqref{eq:A3}. Moreover, for any $\epsilon>0$ we have the uniform in $\epsilon$ energy inequality
\begin{align}
  \Vert \teps(T) \Vert_{L^2(\RRd)}^2 + 2 \int_{0}^{T} \Vert \nabla \teps(s) \Vert_{L^2(\RRd)}^2\; ds\leq \Vert \theta_0 \Vert_{L^2(\RRd)}^2, \label{eq:A4}
\end{align}for any $T>0$, and thus
\begin{align}\label{eq:A5}
 \teps \ \mbox{is bounded in}\ C([0,T];L^2(\RRd))\cap L^2(0,T;\dot{H}^1(\RRd)).
\end{align}

This guarantees that, up to a subsequence, $\teps$ converges weakly to some function $\theta \in L^\infty(0,T;L^2)\cap L^2(0,T;\dot{H}^1)$ (this convergence is weak-$\ast$ in $L^\infty(0,T;L^2)$). This does not suffice to \chd{pass} to the limit in the weak formulation of \eqref{eq:A1}-\eqref{eq:A3}. We next claim that for any compact set ${\mathcal K}\subset \RRd$ we have
\begin{align}\label{eq:A6}
\partial_t \teps \ \mbox{is bounded in}\ L^{4/3}(0,T;W^{-2,\frac{2d}{2d-1}}({\mathcal K})).
\end{align}Indeed, from \eqref{eq:A5}, the Gagliardo-Nirenberg inequality, and interpolation, it follows that $\teps$ is bounded in $L^4(0,T;L^{2d/(d-1)}(\RRd))$. Since $T_{ij}$ are bounded from $L^2(\RRd)$ into itself, by \eqref{eq:A5} it follows that $\ueps$ is bounded in $L^2(0,T;L^2(\RRd))$. Therefore, by H\"older's inequality, $\ddiv(\ueps \teps)$ is bounded in $L^{4/3}(0,T;W^{-1,2d/(2d-1)}(\RRd))$. Lastly, $\epsilon \Lambda^3 \teps$ is bounded in $L^2(0,T;{H}^{-2}(\RRd))$, and $\Delta \teps$ is a bounded family in $L^2(0,T;H^{-1}(\RRd))$. Therefore, by \eqref{eq:A1}, restricting to a compact ${\mathcal K}$, we obtain that $\partial_t \teps$ is bounded in
\begin{align*}
  L^{4/3}(0,T;W^{-1,\frac{2d}{2d-1}}({\mathcal K})) + L^2(0,T;{H}^{-2}({\mathcal K})) + L^2(0,T;H^{-1}({\mathcal K})),
\end{align*}and hence in $L^{4/3}(0,T;W^{-2,\frac{2d}{2d-1}}({\mathcal K}))$ by the Sobolev inequality, proving \eqref{eq:A6}.

Since the injection $H^1({\mathcal K})$ into $L^2({\mathcal K})$ is compact, the injection of $L^2({\mathcal K})$ into $W^{-2,2d/(2d-1)}({\mathcal K})$ is continuous, it follows from the Aubin-Lions compactness lemma \cite[Theorem 3.2.1]{Temam01} (cf.~\cite{LionsJL69}) that
\begin{align}\label{eq:A7}
  \teps \rightarrow \theta \ \mbox{strongly in}\ L^2(0,T;L_{loc}^2(\RRd))
\end{align}since ${\mathcal K}$ was arbitrary. Passing to the limit in the weak formulation of \eqref{eq:A1}-\eqref{eq:A3} is nontrivial only for the nonlinear term. For any $\varphi \in C_{0}^{\infty}( (0,\infty)\times \RRd)$, upon recalling that $u_j = \partial_i T_{ij}[\theta]$, and an integration by parts in $x_i$, we have
\begin{align}
  &{\int\!\!\!\int}\left( \teps \ueps \cdot \nabla \varphi -\theta u \cdot \nabla \varphi\right) \notag \\
  &\qquad = {\int\!\!\!\int} (\teps -\theta) u \cdot \nabla \varphi - {\int\!\!\!\int} \partial_i\teps \,  T_{ij}[\teps -\theta]\, \partial_j \varphi - {\int\!\!\!\int} \teps\,  T_{ij}[\teps -\theta]\, \partial_i\partial_j \varphi \notag \\
  &\qquad = I_\epsilon + II_\epsilon + III_\epsilon.\label{eq:AUX2}
\end{align}Since $u \in \LL{2}{2}$, by \eqref{eq:A7} and the H\"older inequality it follows that $I_\epsilon \rightarrow 0$ as $\epsilon\rightarrow 0$. To obtain the convergence of $II_\epsilon$ and $III_\epsilon$, we claim that
\begin{align}
  T_{ij}[\teps -\theta] \rightarrow 0 \ \mbox{strongly in}\ L^2(0,T;L_{loc}^2(\RRd)) \label{eq:A8}.
\end{align}The proof of \eqref{eq:A8} is similar to that of \eqref{eq:A7}. Since $T_{ij}$ is bounded on $L^2(\RRd)$ and on $\dot{H}^{1}(\RRd)$, it follows from \eqref{eq:A5} that
\begin{align*}
T_{ij}[\teps]\ \mbox{is bounded in}\ C([0,T];L^2(\RRd))\cap L^2(0,T;\dot{H}^1(\RRd))
\end{align*}
Also, $T_{ij}$ is bounded on $L^{2d/(2d-1)}(\RRd)$, so that we obtain $T_{ij}[\ueps\, \teps]$ is bounded in $L^{4/3}(0,T;L^{2d/(2d-1)}(\RRd))$. Fix a compact ${\mathcal K}$ and a test function $\phi$ supported on ${\mathcal K}$. Applying $T_{ij}$ to \eqref{eq:A1}, integrating against $\phi$, and integrating by parts, we obtain
\begin{align*}
  |\langle \partial_t T_{ij}[ \teps], \phi \rangle| &= |\langle T_{ij}[\ueps\, \teps], \nabla \phi \rangle + \langle \nabla T_{ij}[\teps], \nabla \phi \rangle + \epsilon \langle \Lambda T_{ij}[\teps], \Delta \phi \rangle|\\
  &\leq \Vert T_{ij}[\ueps\, \teps] \Vert_{L_{t}^{4/3}L_{x}^{2d/(2d-1)}} \Vert \phi \Vert_{L_{t}^{4}W_{0}^{1,2d}} + \Vert T_{ij}[\teps] \Vert_{\LH{2}} \Vert \phi \Vert_{L_{t}^{2}W_{0}^{2,2}}\\
  & \leq  C \Vert \ueps\, \teps \Vert_{L_{t}^{4/3}L_{x}^{2d/(2d-1)}} \Vert \phi \Vert_{L_{t}^{4}W_{0}^{2,2d}} + C \Vert\teps \Vert_{\LH{2}} \Vert \phi \Vert_{L_{t}^{4}W_{0}^{2,2d}}.
\end{align*}In the last estimate we have also used the H\"older and Poincar\'e inequalities. The above proves that
\begin{align*}
\partial_t T_{ij}[\teps]\ \mbox{is bounded in}\ L^{4/3}(0,T;W^{-2,\frac{2d}{2d-1}}({\mathcal K})).
\end{align*}The claim \eqref{eq:A8} now follows directly from the Aubin-Lions lemma (cf.~\cite{LionsJL69,Temam01}). Moreover, this shows that in \eqref{eq:A8} we have $III_\epsilon \rightarrow 0$ and $II_\epsilon \rightarrow 0$ as $\epsilon \rightarrow 0$.

This proves that $\theta$ is a weak solution to the limit system, i.e., \eqref{eq:N1}-\eqref{eq:N4}. By construction it satisfies the energy inequality, concluding the proof of existence of the Leray-Hopf weak solutions to \eqref{eq:N1}-\eqref{eq:N4}.

\subsection*{Acknowledgements} We would like to thank Igor Kukavica for fruitful discussions on the De Giorgi method, and Nata\v{s}a Pavlovi\'c for very helpful discussions on earlier versions of this draft. The work of S.F. is supported by the NSF grant DMS 0803268. The work of V.V. was in part supported by the NSF grant DMS 1009769. S.F thanks DAMTP and Trinity College, Cambridge, for their kind hospitality at the time this paper was initiated.


\begin{thebibliography}{99}

\bibitem{AronsonSerrin} D.G.~Aronson and J.~Serrin, {\em Local behavior of solutions of quasilinear parabolic equations}. Arch.~Rational Mech.~Anal.~{\bf 25} (1967), 81--122.

\bibitem{CaffVass} L.~Caffarelli and  A.~Vasseur, {\em Drift diffusion equations with fractional diffusion and the quasi-geostrophic equation}. Annals of Mathematics~{\bf 171} (2010), No.~3, 1903–-1930.

\bibitem{CaffVass2} L.~Caffarelli and  A.~Vasseur, {\em The De Giorgi method for regularity of solutions of elliptic equations and its applications to fluid dynamics}.
Discrete Contin.~Dyn.~Syst.~Ser.~S~{\bf 3} (2010), no.~3, 409--427.

\bibitem{CannonePlanchon} M.~Cannone, F.~Planchon, {\em More Lyapunov functions for the Navier-Stokes equations}, in {\em Navier-Stokes equations: Theory and Numerical Methods}. R. Salvi, ed., Lecture Notes in Pure and Applied Mathematics~{\bf 223}, New York-Oxford (2001), 19-–26.

\chd{\bibitem{CheminLerner} J.-Y.~Chemin and N.~Lerner, {\em Flot de champs de vecteurs non lipschitziens et \'equations de Navier-Stokes}. J.~Differential Equations~{\bf 121} (1995), no.~2, 314--328.}

\bibitem{CorCor04} A.~C\'ordoba, D.~C\'ordoba, {\em A maximum principle applied to quasi-geostrophic equations}. Comm. Math. Phys. {\bf 249} (2004),  no. 3, 511--528.

\bibitem{CorFeff} D.~C\'ordoba and C.~Fefferman, {\em Growth of solutions for QG and 2D Euler equations}. J.~Am.~Math.~Soc.~{\bf 15} (2002), no.~3, 665--670.

\bibitem{ConstIyerWu} P.~Constantin, G.~Iyer, and J.~Wu, {\em Global regularity for a modified critical dissipative quasi-geostrophic equation}.  Indiana Univ.~Math.~J.~{\bf 57} (2008),  no.~6, 2681--2692.

\bibitem{ConstMajTab} P.~Constantin, A.~J.~Majda, E.~Tabak, {\em Formation of strong fronts in the 2-D quasi-geostrophic thermal active scalar}. Nonlinearity {\bf 7} (1994), no.~6, 1495--1533.


\bibitem{ConstWu08} P.~Constantin and J.~Wu, {\em Regularity of H\"older continuous solutions of the supercritical quasi-geostrophic equation}. Ann.~Inst.~H.~Poincar\'e Anal.~Non Lin\'eaire~{\bf~25}  (2008),  no.~6, 1103--1110.

\bibitem{ConstWu09} P.~Constantin and J.~Wu, {\em H\"older continuity of solutions of supercritical dissipative hydrodynamic transport equations}. Ann.~Inst.~H.~Poincar\'e Anal.~Non Lin\'eaire~{\bf 26} (2009),  no.~1, 159--180.

\bibitem{DeGiorgi} E.~De~Giorgi, {\em Sulla differenziabilit\`a e l'analiticit\`a delle estremali degli integrali multipli regolari}. Mem.~Accad.~Sci.~Torino. Cl.~Sci.~Fis.~Mat.~Nat.~{\bf 3} (1957), 3:25–-43.

%\bibitem{DongKim} H.~Dong and D.~Kim, {\em Elliptic equations in divergence form with partially BMO coefficients}, Arch.~Rational Mech.~Anal.~{\bf 196} (2010), no.~1, 25--70.

\bibitem{FV-higher} S.~Friedlander and V.~Vicol, {\em Higher regularity of H\"older continuous solutions of parabolic equations with singular drift velocities}. arXiv:1102.0585v1 [math.AP].

\bibitem{Giaq} M.~Giaquinta, {\em Introduction to regularity theory for nonlinear elliptic systems}. Lectures in Mathematics ETH Zürich. Birkh\"auser Verlag, Basel, 1993.

\bibitem{Glatzmaier} G.A.~Glatzmaier, D.E.~Ogden, and T.L.~Clune, {\em Modeling the Earth's Dynamo} in {\em State of the Planet: Frontiers and Challenges in Geophysics}. Geophysical Monograph~150 (2004), eds R.S.J.~Sparks, C.J.~Hawkesworth, IUGG~{\bf 19}, 13--24.

\bibitem{KisNazVolb} A.~Kiselev, F.~Nazarov, and  A.~Volberg, {\em Global well-posedness for the critical 2D dissipative quasi-geostrophic equation}. Invent.~Math.~{\bf 167} (2007), no.~3, 445--453.

\bibitem{KochTat} H.~Koch and D.~Tataru, {\em Well Posedness for the Navier--Stokes
equations}. Adv. Math.~{\bf 157} (2001), 22–-35.

\bibitem{LadySolonnUralceva} O.A.~Lady\v{z}enskaja, V.A.~Solonnikov, and N.N.~Ural'ceva, {\em
Linear and quasilinear equations of parabolic type}. (Russian) Translated from the Russian by S.~Smith. Translations of Mathematical Monographs, Vol.~{\bf 23} American Mathematical Society, Providence, RI, 1967.

\bibitem{Lieberman} G.M.~Lieberman, {\em Second order parabolic differential equations}. World Scientific Publishing Co., Inc., River Edge, NJ, 1996.

\bibitem{LionsJL69} J.-L.~Lions,{\em Quelque M\'ethodes de R\'esolutions des Probl\'emes aux Limites Non-Lin\'eares}. Dunod, Paris, 1969.

\bibitem{McLean} W.~McLean, {\em Local and global descriptions of periodic pseudodifferential operators}. Math.~Nachr.~{\bf 150} (1991), 151–-161.

\bibitem{Moffatt} H.K.~Moffatt, {\em Magnetostrophic turbulence and the geodynamo}. IUTAM Symposium on Computational Physics and New Perspectives in Turbulence, 339--346, IUTAM Bookser., 4, Springer, Dordrecht, 2008.

\bibitem{Moser} J.~Moser, {\em A Harnack inequality for parabolic differential equations}. Commun.~Pure Appl.~Math.~{\bf 17} (1964), 101--134.

\bibitem{Nash} J.~Nash, {\em Continuity of solutions of parabolic and elliptic equations}. Amer.~J.~Math.~{\bf 80} (1958), 931--954.

\bibitem{Osada} H.~Osada, {\em Diffusion processes with generators of generalized divergence form}. J.~Math.~Kyoto Univ.~{\bf 27} (1987),  no.~4, 597--619.

\bibitem{Ruzhansky} M.~Ruzhansky, V.~Turunen, {\em On the toroidal quantization of periodic pseudo-differential operators}.  Numer.~Funct.~Anal.~Optim.~{\bf 30}  (2009),  no.~9-10, 1098--1124.

\bibitem{Semenov} Y.A.~Semenov, {\em Regularity theorems for parabolic equations}. J.~Funct.~Anal.~{\bf 231} (2006), no.~2, 375--417.

\bibitem{SSSZ} G.~Seregin, L.~Silvestre, V.~\v{S}ver\'ak, and A.~Zlato\v{s}, {\em On divergence-free drifts}. arXiv:1010.6025v1 [math.AP]

\bibitem{Silvestre10a} L.~Silvestre, {\em Eventual regularization for the slightly supercritical quasi-geostrophic equation}. Ann.~Inst.~H.~Poincar\'e Anal.~Non Lin\'eaire~{\bf 27} (2010), no.~2, 693--704.

%\bibitem{Silvestre10b} L.~Silvestre, {\em H\"older estimates for advection fractional-diffusion equations}. arXiv:1009.5723v1 [math.AP]

\bibitem{Stein93} E.M.~Stein, {\em Harmonic analysis: real-variable methods, orthogonality, and oscillatory integrals}. Princeton Mathematical Series~{\bf 43}, Princeton, NJ, Princeton University Press, 1993.

\bibitem{Temam01} R.~Temam, {\em Navier-Stokes equations. Theory and numerical analysis.} Reprint of the 1984 edition. AMS Chelsea Publishing, Providence, RI, 2001.

\bibitem{Vass07} A.~Vasseur, {\em A new proof of partial regularity of solutions to Navier-Stokes equations}. NoDEA Nonlinear Differential Equations Appl.~{\bf 14} (2007), no.~5-6, 753--785.

\bibitem{Wu04} J.~Wu, {\em Global solutions of the 2D dissipative quasi-geostrophic equation in Besov spaces}. SIAM J.~Math.~Anal.~{\bf 36} (2004), no.~3, 1014--1030.

\bibitem{Zhang04} Q.S.~Zhang, {\em A strong regularity result for parabolic equations}. Commun.~Math.~Phys.~{\bf 244} (2004), 245--260.

\bibitem{Zhang06} Q.S.~Zhang, {\em Local Estimates on Two Linear Parabolic Equations with Singular Coefficients}. Pacific Journal of Math.~{\bf 223} (2006), no.~2, 367--396.

\end{thebibliography}
\end{document}